\documentclass[reqno]{amsart}
\usepackage{amsmath}
\usepackage{amsthm, amscd, amssymb, amsfonts, amsbsy}
\usepackage[usenames, dvipsnames]{color}
\usepackage{color}
\usepackage{enumerate}
\usepackage{verbatim}

\numberwithin{equation}{section}

\theoremstyle{plain}
\newtheorem{theorem}{Theorem}[section]
\newtheorem{lemma}[theorem]{Lemma}

\theoremstyle{definition}
\newtheorem{definition}[theorem]{Definition}
\newtheorem{assumption}[theorem]{Assumption}

\theoremstyle{remark}
\newtheorem{remark}{Remark}[section]

\makeatletter
\def\dashint{\operatorname%
{\,\,\text{\bf--}\kern-.98em\DOTSI\intop\ilimits@\!\!}}
\makeatother

\def\bR{\mathbb{R}}
\def\cL{\mathcal{L}}
\def\cB{\mathcal{B}}
\def\cD{\mathcal{D}}

\begin{document}
\title[Green function]{Green functions of conormal derivative problems for stationary Stokes system}

\author[J. Choi]{Jongkeun Choi}
\address[J. Choi]{Division of Applied Mathematics, Brown University, 182 George Street, Providence, RI 02912, USA}
\email{Jongkeun\_Choi@brown.edu}
\thanks{J. Choi was supported by Basic Science Research Program
through the National Research Foundation of Korea (NRF) funded by the Ministry of Education
(2017R1A6A3A03005168)}

\author[H. Dong]{Hongjie Dong}
\address[H. Dong]{Division of Applied Mathematics, Brown University, 182 George Street, Providence, RI 02912, USA}
\email{Hongjie\_Dong@brown.edu}

\thanks{H. Dong was partially supported by the NSF under agreement DMS-1600593}

\author[D. Kim]{Doyoon Kim}
\address[D. Kim]{Department of Mathematics, Korea University, 145 Anam-ro, Seongbuk-gu, Seoul, 02841, Republic of Korea}
\email{doyoon\_kim@korea.ac.kr}

\thanks{D. Kim was supported by Basic Science Research Program through the National Research Foundation of Korea (NRF) funded by the Ministry of Education (2016R1D1A1B03934369)}


\begin{abstract}
We study Green functions for stationary Stokes systems satisfying the conormal derivative boundary condition.
We establish existence, uniqueness, and various estimates for the Green function under the assumption that weak solutions of the Stokes system are continuous in the interior of the domain.
Also, we establish the global pointwise bound for the Green function under the additional assumption that weak solutions of the conormal derivative problem for the Stokes system are locally bounded up to the boundary.
We provide some examples satisfying such continuity and boundedness properties.
\end{abstract}

\maketitle

\section{Introduction}		\label{S1}

In this paper, we are concerned with construction and pointwise estimates for Green functions for stationary Stokes systems with the conormal derivative boundary condition in a bounded domain $\Omega\subset \bR^d$, $d\ge 3$.
Let $\cL$ be a second-order elliptic operator in divergence form
$$
\cL u=D_\alpha(A^{\alpha\beta}D_\beta u)
$$
acting on column vector valued functions $u=(u^1,\ldots,u^d)^{\top}$.
The coefficients $ A^{\alpha\beta}= A^{\alpha\beta}(x)$ are $d\times d$ matrix-valued functions on $\bR^d$ with entries $A_{ij}^{\alpha\beta}$ satisfying the strong ellipticity condition;
i.e., there is a constant $\lambda\in (0,1]$ such that
\begin{equation}		\label{171209@eq1}
|A^{\alpha\beta}|\le \lambda^{-1}, \quad \sum_{\alpha,\beta=1}^dA^{\alpha\beta}\xi_\beta\cdot \xi_\alpha\ge \lambda \sum_{\alpha=1}^d|\xi_\alpha|^2
\end{equation}
for any $\xi_\alpha\in \bR^d$, $\alpha\in \{1,\ldots,d\}$.
By a Green function with conormal derivative boundary condition for the Stokes system, we mean a pair $(G, \Pi)=(G(x,y),\Pi(x,y))$, where $G$ is a $d\times d$ matrix-valued function and $\Pi$ is a $1\times d$ vector-valued function, which is a solution of the conormal derivative problem
$$
\left\{
\begin{aligned}
\operatorname{div} G(\cdot,y)=0 &\quad \text{in }\, \Omega,\\
\cL G(\cdot,y)+\nabla \Pi(\cdot,y)=-\delta_yI+\frac{1}{|\Omega|}I &\quad \text{in }\, \Omega,\\
\cB G(\cdot,y)+{\nu}\Pi(\cdot,y)=0 &\quad \text{on }\, \partial \Omega,
\end{aligned}
\right.
$$
where $\nu=(\nu_1,\ldots,\nu_d)^{\top}$ is the outward unit normal to $\partial \Omega$, $\cB G(\cdot,y)$ is the conormal derivative of $G(\cdot,y)$ on $\partial \Omega$ associated with $\cL$, $\delta_y$ is the Dirac delta function concentrated at $y$, and $I$ is the $d\times d$ identity matrix; see Section \ref{S2} for a more precise definition of the Green function.
The conormal derivative boundary condition arises from the variational principle, and one may consider this type of condition on the output of the channel where the velocity of the flow is a prior unknown when describing a flow through a finite channel. See \cite{MR1972357,
MR3458302} and references therein.

Our focus in this paper is to find minimal regularity assumptions on the coefficients and on the boundary of the domain for the existence of the Green function $(G, \Pi)$ satisfying the pointwise bound
\begin{equation}		\label{171207@eq1}
|G(x,y)|\le C|x-y|^{2-d},   \quad x\neq y.
\end{equation}
We prove that if the Poincar\'e inequality \eqref{170602@eq2} holds in $\Omega$ and weak solutions of the Stokes system
with bounded data are continuous in the interior of the domain,
then the Green function exists and satisfies the pointwise bound \eqref{171207@eq1} away from the boundary of the domain; see Theorem \ref{T1}.
The Green function satisfies the pointwise bound \eqref{171207@eq1}
globally if we further assume that weak solutions of the Stokes system with bounded data are locally bounded up to the boundary; see Theorem \ref{T2}.
For the uniqueness of Green functions, we impose the normalization condition
\begin{equation}		\label{171212@eq1}
\int_\Omega G(x,y)\,dx=0,
\end{equation}
which enables us to construct the Green function in a large class of domains.
One can investigate the Green function with the normalization condition
\begin{equation}		\label{171212@eq2}
\int_{\partial \Omega} G(x,y)\,dx=0
\end{equation}
if the boundary trace of a $W^1_2(\Omega)$ function is well defined.
See \cite{MR3017032} for the Neumann Green functions of elliptic and parabolic systems with the normalization condition \eqref{171212@eq1} and \cite{MR3105752} for that of elliptic systems with the normalization condition \eqref{171212@eq2}.

An easy consequence of our results combined with the $L_q$-estimates for the Stokes system established in \cite{MR3809039} is the following.
If $A^{\alpha\beta}$ are merely measurable in one direction, which may differ depending on the local coordinates, and have small mean oscillations in the other directions (variably partially BMO) and the domain is Reifenberg flat, then the Green function exists and satisfies the global pointwise bound \eqref{171207@eq1}; see Appendix.
We note that Stokes systems with such type of variable coefficients can be used to describe the motion of inhomogeneous fluids with density dependent viscosity and two fluids with interfacial boundaries;
see \cite{arXiv:1604.02690v2, MR3758532} and the references therein.
It also can occur when performing a change of coordinates or when flattening the boundary; see \cite{MR3670039}.

As such, Stokes systems with variable coefficients were discussed in many papers. With regard to the Green function of the Dirichlet problem, we refer the recent papers \cite{MR3693868, MR3670039}.
In \cite{MR3693868}, the authors proved the existence and pointwise estimates of Green functions in a bounded $C^1$ domain when $d\ge 3$ and coefficients have vanishing mean oscillations.
The corresponding results for the fundamental solution and the Green function on a half space were obtained in \cite{MR3670039} when coefficients are merely measurable in one direction and have small mean oscillations in the other directions (partially BMO).
See also \cite{arXiv:1710.05383v2} for the asymptotic behavior of the Green function for Stokes system with oscillating periodic coefficients.
Regarding the regularity theory, we refer the reader to \cite{arXiv:1604.02690v2, MR3758532, MR3809039} for Stokes system with variably partially BMO coefficients, and \cite{MR3693868, MR3614614} for Stokes system with coefficients having small mean oscillations in all directions.
See also \cite{MR3505170} for the homogenization theory of Stokes system with periodic coefficients.
On the other hand, we are unable to find any literature explicitly dealing with the Green function satisfying the conormal derivative boundary condition for Stokes system with variable coefficients.
With respect to the classical Stokes system with the Laplace operator
$$
\Delta u+\nabla p=f,
$$
the Neumann Green function was studied by Maz'ya-Rossmann.
In \cite{MR2182091}, they obtained estimates for the Green function of a mixed boundary value problem (containing the Neumann boundary value problem) on a polyhedral cone in $\bR^3$.
See also \cite{MR3320459} for the Green function of the mixed problem in a two dimensional domain.
Regarding the Green function for the Dirichlet problem, we refer the reader to Maz'ya-Plamensvski\u \i~ \cite{MR0725151}, where the authors proved the existence and the pointwise bound for the Green function in three-dimensional domains of polyhedral type.
For this line of research, see \cite{MR2676181, MR2808700} and the references therein.
We also refer to \cite{MR2763343} for the Green function on a Lipschitz domain in $\bR^2$ or $\bR^3$, and \cite{MR1683625, MR2808162} for the fundamental solution of the Stokes system.

Our argument in establishing the existence of the Green function is based on techniques used in Gr\"uter-Widman \cite{MR0657523} and Hofmann-Kim \cite{MR2341783}, where the authors constructed Green functions for elliptic equations and systems with irregular coefficients.
The key for obtaining the Green function lies in constructing a sequence of approximated Green functions, getting uniform estimates for the sequence, and applying a compactness theorem.
In this paper, to establish the existence of the Green function $(G,\Pi)$,  we refine the techniques for the uniform estimates since the presence of the pressure term $\Pi$ makes the argument more involved.
For the global pointwise bound \eqref{171207@eq1}, we adapt the argument in Kang-Kim \cite{MR2718661}, where the authors proved global pointwise bounds of Green functions for elliptic systems.

In a subsequent paper,
we will study Green functions $(G,\Pi)$ for Stokes systems with measurable coefficients in two dimensional domains.
In this case, $G$ should have a logarithmic growth.
As a matter of fact, our method breaks down and is not applicable in the two dimensional case.

The remainder of this paper is organized as follows.
In Section \ref{S2}, we state our main results along with some notation, assumptions, and the definition of the Green function.
We provide some auxiliary results in Section \ref{S3}.
In Sections \ref{S4} and \ref{S5}, we prove the main theorems.
Finally we provide some applications of our results in Appendix.

\section{Main results}		\label{S2}

Throughout the paper,
we denote by $\Omega$ a bounded domain in the Euclidean space $\bR^d$, $d\ge 3$.
For any $x\in \Omega$ and $r>0$, we write $\Omega_r(x)=\Omega \cap B_r(x)$, where $B_r(x)$ is the usual Euclidean ball of radius $r$ centered at $x$.
For $1\le q\le \infty$, we denote by $W^1_q(\Omega)$ the usual Sobolev space.
We define
$$
\tilde{L}_q(\Omega)=\{u\in L_q(\Omega):(u)_\Omega=0\}, \quad
\tilde{W}^1_{q}(\Omega)=\{u\in W^1_{q}(\Omega): (u)_{\Omega}=0\},
$$
where $(u)_\Omega$ is the average of $u$ in $\Omega$, i.e.,
\[
(u)_\Omega=\dashint_\Omega u\,dx=\frac{1}{|\Omega|}\int_\Omega u\,dx.
\]

Let $\cL$ be an elliptic operator in divergence form
$$
\cL u=D_\alpha(A^{\alpha\beta}D_\beta u),
$$
where the coefficients $A^{\alpha\beta}$ satisfy the strong ellipticity condition \eqref{171209@eq1}.
We denote by $\cB u=A^{\alpha\beta} D_\beta u \nu_\alpha$ the conormal derivative of $u$ on the boundary of $\Omega$ associated with $\cL$.
The adjoint operator $\cL^*$ and the conormal derivative operator $\cB^*$ associated with $\cL^*$ are defined by
$$
\cL^* u=D_\alpha((A^{\beta\alpha})^\top D_\beta u), \quad \cB^* u=(A^{\beta\alpha})^{\top} D_\beta u \,\nu_\alpha.
$$
For $f\in\tilde{L}_{2d/(d+2)}(\Omega)^d$ and $f_\alpha \in L_{2}(\Omega)^d$,  we say that $(u, p)\in W^1_{2}(\Omega)^d\times L_{2}(\Omega)$ is a weak solution of the problem
$$
\left\{
\begin{aligned}
\cL u+\nabla p= f+D_\alpha  f_\alpha \quad &\text{in }\ \Omega,\\
\cB u+p\nu = f_\alpha \nu_\alpha \quad &\text{on }\ \partial \Omega,
\end{aligned}
\right.
$$
if
\[
\int_{\Omega}  A^{\alpha\beta}D_\beta  u\cdot D_\alpha  \phi\ dx+\int_\Omega p\, \operatorname{div}  \phi\ dx=-\int_{\Omega} f\cdot \phi\ dx+\int_\Omega  f_\alpha\cdot D_\alpha \phi \ dx
\]
holds for any $\phi\in W^1_2(\Omega)^d$.
Similarly, we say that $(u, p)\in W^1_{2}(\Omega)^d\times L_{2}(\Omega)$ is a weak solution of the adjoint problem
$$
\left\{
\begin{aligned}
\cL^* u+\nabla p= f+D_\alpha  f_\alpha \quad &\text{in }\ \Omega,\\
\cB^* u+p\nu = f_\alpha \nu_\alpha \quad &\text{on }\ \partial \Omega,
\end{aligned}
\right.
$$
if
\[
\int_{\Omega}  A^{\alpha\beta}D_\beta  \phi \cdot D_\alpha  u\ dx+\int_\Omega p \operatorname{div}  \phi\ dx=-\int_{\Omega} f\cdot \phi\ dx+\int_\Omega  f_\alpha\cdot D_\alpha \phi \ dx
\]
holds for any $\phi\in W^{1}_2(\Omega)^d$.
For $u\in W^{1}_1(\Omega)^d$ and $g\in L_1(\Omega)^d$,
by
$$
\operatorname{div} u=g \quad \text{in }\, \Omega
$$
we mean the equation holds in the almost everywhere sense.

In the definition below, $G=G(x,y)$ is a $d\times d$ matrix-valued function with the entries $G^{ij}:\Omega\times \Omega\to [-\infty,\infty]$, and $\Pi=\Pi(x,y)$ is a ${1\times d}$ vector-valued function with the entries $\Pi^i:\Omega\times \Omega \to [-\infty,\infty]$.

\begin{definition}[Green function]		\label{D1}
We say that $(G,\Pi)$ is a Green function of $\cL$ in a bounded domain $\Omega$ if it satisfies the following properties:
\begin{enumerate}[$(a)$]
\item
For any $y\in \Omega$ and $r>0$,
$$
G(\cdot,y)\in \tilde{W}^1_1(\Omega)^{d\times d}\cap W^1_2(\Omega\setminus B_r(y))^{d\times d},
$$
$$
\Pi(\cdot,y)\in L_1(\Omega)^d\cap L_2(\Omega\setminus B_r(y))^d.
$$
\item
For any $y\in \Omega$, $(G(\cdot,y),\Pi(\cdot,y))$ satisfies
$$
\left\{
\begin{aligned}
\operatorname{div} G(\cdot,y)=0 \quad &\text{in }\, \Omega,\\
\cL G(\cdot,y)+\nabla \Pi(\cdot,y)=-\delta_y I+\frac{1}{|\Omega|}I \quad &\text{in }\ \Omega,\\
\cB G(\cdot,y)+{\nu}\Pi(\cdot,y)=0 \quad&\text{on }\ \partial \Omega,
\end{aligned}
\right.
$$
in the sense that, for any $k\in \{1,\ldots,d\}$ and $\phi\in W^1_\infty(\Omega)^d\cap C(\Omega)^d$, we have
$$
\operatorname{div} G^{\cdot k}(\cdot,y)=0 \quad \text{in }\, \Omega
$$
and
$$
\int_\Omega A^{\alpha\beta}D_\beta G^{\cdot k}(\cdot,y)\cdot D_\alpha \phi\,dx+\int_\Omega \Pi^k(\cdot,y) \operatorname{div}\phi\,dx=\phi^k(y)-\dashint_\Omega \phi^k\,dx,
$$
where $G^{\cdot k}(\cdot,y)$ is the $k$-th column of $G(\cdot,y)$.
\item
Let $f\in \tilde{L}_\infty(\Omega)^d$ and $g\in L_\infty(\Omega)$.
If $(u,p)\in \tilde{W}^1_2(\Omega)^d\times L_2(\Omega)$ is a weak solution of
\begin{equation}		\label{170129@eq3}
\left\{
\begin{aligned}
\operatorname{div}u=g \quad &\text{in }\ \Omega,\\
\cL^* u+\nabla p= f\quad &\text{in }\ \Omega,\\
\cB^*u+p\nu =0 \quad &\text{on }\ \partial \Omega,
\end{aligned}
\right.
\end{equation}
then for a.e. $y\in \Omega$, we have
\begin{equation}		\label{180123@eq1}
u(y)=-\int_\Omega G(x,y)^{\top}f(x)\,dx+\int_{\Omega} \Pi(x,y)^{\top}g(x)\,dx.
\end{equation}
\end{enumerate}
\end{definition}

We remark that the property (c) in the above definition together with the solvability of the conormal derivative problem in $\tilde{W}^1_2(\Omega)^d\times L_2(\Omega)$ (see Lemma \ref{170602@lem2}) gives the uniqueness of a Green function in the sense that, if $(\tilde{G},\tilde{\Pi})$ is another Green function satisfying the above properties, then
for each $\phi \in C_0^\infty(\Omega)^d$ and $\varphi \in C_0^\infty(\Omega)$, we have
$$
\int_\Omega \big( G(x,y)^{\top} - \tilde{G}(x,y)^{\top} \big)  \phi(x) \, dx = \int_\Omega \big( \Pi(x,y)^{\top} - \tilde{\Pi}(x,y)^{\top} \big) \varphi(x) \, dx = 0
$$
for a.e. $y \in \Omega$.
We also note that the definition of the Green function depends on the {\em{normalization condition}}.
In the above definition, the Green function satisfies
$\int_\Omega G(x,y)\,dx=0$.
On the other hand, if $\Omega$ is a Lipschitz domain so that the boundary trace of a $W^1_2(\Omega)$ function is well defined, then one may use the normalization condition $\int_{\partial \Omega}G(x,y)\,d\sigma_x=0$.
Under this condition, the Green function $(G,\Pi)$ can be defined as a solution of the problem
$$
\left\{
\begin{aligned}
\operatorname{div} G(\cdot,y)=0 &\quad \text{in }\, \Omega,\\
\cL G(\cdot,y)+\nabla \Pi(\cdot,y)=-\delta_y I &\quad \text{in }\ \Omega,\\
\cB G(\cdot,y)+{\nu}\Pi(\cdot,y)=-\frac{1}{|\partial\Omega|}I &\quad \text{on }\ \partial \Omega.
\end{aligned}
\right.
$$

We make the following assumptions to construct the Green function for $\cL$ in $\Omega$.

\begin{assumption}		\label{A0}
There exists a constant $K_0>0$ satisfying
\begin{equation}		\label{170602@eq2}
\|\phi\|_{L_{2d/(d-2)}(\Omega)}\le K_0\|D\phi\|_{L_2(\Omega)} \quad \text{for all }\, \phi\in \tilde{W}^1_2(\Omega).
\end{equation}
\end{assumption}

\begin{remark}
By following the argument in \cite[pp. 286--290]{MR2597943}, one can show that
 Assumption \ref{A0} holds when $\Omega$ is an extension domain, in particular, when
$\Omega$ is a Reifenberg flat domain as in Assumption \ref{170219@ass2} $(ii)$; see, for instance, \cite{MR0631089, MR3186805}.
\end{remark}

The following assumption holds, for instance, when the coefficients $A^{\alpha\beta}$ are variably partially BMO; see Appendix.

\begin{assumption}		\label{A1}
There exist constants $R_0\in (0,1]$ and $A_0>0$ such that the following holds:
Let $x_0\in \Omega$ and $0<R<\min\{R_0, \operatorname{dist}(x_0, \partial \Omega)\}$.
If $(u,p)\in W^1_2(B_R(x_0))^d\times L_2(B_R(x_0))$
satisfies
\begin{equation}		\label{171129@eq1}
\left\{
\begin{aligned}
\operatorname{div}u=0 &\quad \text{in }\ B_R(x_0),\\
\cL u+\nabla p=f &\quad \text{in }\ B_R(x_0),
\end{aligned}
\right.
\end{equation}
where $f\in L_\infty(B_R(x_0))^d$, then we have
$u\in C(\overline{B_{R/2}(x_0)})^d$ (in fact, a version of $u$ belongs to $C(\overline{B_{R/2}(x_0)})^d)$ with the estimate
$$
\|u\|_{L_\infty(B_{R/2}(x_0))}\le A_0 \big(R^{-d/2}\|u\|_{L_2(B_R(x_0))}+R^2\|f\|_{L_\infty(B_R(x_0))}\big).
$$
The same statement holds true when $\cL$ is replaced by $\cL^*$.
\end{assumption}

\begin{theorem} 		\label{T1}
Let $d\ge 3$ and $\Omega$ be a bounded domain in $\bR^d$.
Then under Assumptions \ref{A0} and \ref{A1}, there exist Green functions $(G,\Pi)$ of $\cL$ and $(G^*, \Pi^*)$ of $\cL^*$
such that for any $y\in \Omega$, we have
$$
G(\cdot,y),\, G^*(\cdot,y) \in C(\Omega\setminus \{y\})^{d\times d},
$$
and there exists a measure zero set $N_y\subset \Omega$ such that
\begin{equation}		\label{170702@eq2}
G(x,y)=G^*(y,x)^{\top}, \quad G(y,x)=G^*(x,y)^\top \quad \text{for all $x\in \Omega\setminus N_y$}.
\end{equation}
Moreover, for any $x,y\in \Omega$ satisfying
$$
0<|x-y|<\frac{1}{2}\min\{R_0, \operatorname{dist}(y, \partial \Omega)\},
$$
we have
\begin{equation}		\label{170702@eq1}
|G(x,y)|\le C|x-y|^{2-d},
\end{equation}
where $C=C(d,\lambda,K_0,A_0)$.
Furthermore,
the following estimates hold for all $y\in \Omega$ and $0<R<\min\{R_0,\operatorname{dist}(y,\partial \Omega)\}$:
\begin{enumerate}[$i)$]
\item
$\|G(\cdot,y)\|_{L_{2d/(d-2)}(\Omega\setminus B_R(y))}+\|DG(\cdot,y)\|_{L_2(\Omega\setminus B_R(y))}\le CR^{(2-d)/2}$.
\item
$\|\Pi(\cdot,y)\|_{L_2(\Omega\setminus B_R(y))}\le CR^{(2-d)/2}$.
\item
$\big|\{x\in \Omega:|G(x,y)|>t\}\big|\le C t^{-d/(d-2)}$ for all $t> \min\{R_0,\operatorname{dist}(y,\partial \Omega)\}^{2-d}$.
\item
$\big|\{x\in \Omega:|D_xG(x,y)|>t\}\big|\le C t^{-d/(d-1)}$ for all $t> \min\{R_0,\operatorname{dist}(y,\partial \Omega)\}^{1-d}$.
\item
$\big|\{x\in \Omega:|\Pi(x,y)|>t\}\big|\le C t^{-d/(d-1)}$ for all $t> \min\{R_0,\operatorname{dist}(y,\partial \Omega)\}^{1-d}$.
\item
$\|G(\cdot,y)\|_{L_q(B_R(y))}\le C_qR^{2-d+d/q}$, where $q\in [1,d/(d-2))$.
\item
$\|DG(\cdot,y)\|_{L_q(B_R(y))}\le C_q R^{1-d+d/q}$, where $q\in [1,d/(d-1))$.
\item
$\|\Pi(\cdot,y)\|_{L_q(B_R(y))}\le C_q R^{1-d+d/q}$, where $q\in [1,d/(d-1))$.
\end{enumerate}
In the above, the constant $C$ depends only on $d$, $\lambda$, $K_0$, and $A_0$, and $C_q$ depends also on $q$.
\end{theorem}

\begin{remark}
In Theorem \ref{T1}, one may prove the existence of the Green functions without the continuity condition in Assumption \ref{A1}.
In this case, the continuity of the Green functions and the identities in \eqref{170702@eq2} are not guaranteed, and the estimate \eqref{170702@eq1} is satisfied in the almost everywhere sense.
On the other hand, if one has a modulus of continuity estimate such as a H\"older estimate for the solution $u$ in Assumption \ref{A1}, then one can show that \eqref{170702@eq2} holds for all $x,y\in \Omega$ with $x\neq y$.
For further details, see the end of the proof of Theorem \ref{T1}.
\end{remark}

\begin{remark}
Let $(G,\Pi)$ and $(G^*,\Pi^*)$ be Green functions for $\cL$ and $\cL^*$, respectively, constructed in Theorem \ref{T1}.
Assume that  $(u,p)\in \tilde{W}^1_2(\Omega)^d\times L_2(\Omega)$ is a weak solution of
$$
\left\{
\begin{aligned}
\operatorname{div}u=g \quad &\text{in }\ \Omega,\\
\cL u+\nabla p= f+D_\alpha f_\alpha\quad &\text{in }\ \Omega,\\
\cB u+p\nu =f_\alpha \nu_\alpha \quad &\text{on }\ \partial \Omega,
\end{aligned}
\right.
$$
where $f\in \tilde{L}_\infty(\Omega)^d$, $f_\alpha\in L_\infty(\Omega)^d$, and $g\in L_\infty(\Omega)$.
Then for a.e. $y\in \Omega$, it holds that
$$
u(y)=-\int_\Omega G^*(x,y)^{\top}f(x)\,dx+\int_\Omega D_\alpha G^*(x,y)^{\top}f_\alpha(x)\,dx+\int_\Omega \Pi^*(x,y)^{\top}g(x)\,dx.
$$
Using this together with \eqref{170702@eq2}, we have
$$
u(y)=-\int_\Omega G(y,x)f(x)\,dx+\int_\Omega D_\alpha G(y,x)f_\alpha(x)\,dx+\int_\Omega \Pi^*(x,y)^{\top}g(x)\,dx.
$$
\end{remark}

To obtain the global pointwise bound for $G(x,y)$, we impose the following assumption.
We note that the assumption holds, for instance, when the coefficients $A^{\alpha\beta}$ of $\cL$ are variably partially BMO and $\Omega$ is a Reifenberg flat domain; see Appendix.

\begin{assumption}		\label{A2}
There exist constants $R_0\in (0,1]$ and $A_1>0$ such that the following holds:
Let $x_0\in \partial \Omega$ and $0<R<R_0$.
If $(u,p)\in {W}^1_2(\Omega_R(x_0))^d\times L_2(\Omega_R(x_0))$ satisfies
\begin{equation}		\label{171129@eq1a}
\left\{
\begin{aligned}
\operatorname{div}u=0 \quad &\text{in }\ \Omega_R(x_0),\\
\cL u+\nabla p= f\quad &\text{in }\ \Omega_R(x_0),\\
\cB u+p\nu =0 \quad &\text{on }\ \partial \Omega\cap B_R(x_0),
\end{aligned}
\right.
\end{equation}
where $f\in {L}_\infty(\Omega_R(x_0))^d$, then we have
$$
\|u\|_{L_\infty(\Omega_{R/2}(x_0))}\le A_1\big(R^{-d/2}\|u\|_{L_2(\Omega_R(x_0))}+R^2\|f\|_{L_\infty(\Omega_R(x_0))}\big).
$$
The same statement holds true when $\cL$ is replaced by $\cL^*$.
\end{assumption}

\begin{remark}
In the above assumption and throughout the paper, $(u,p)$ is said to satisfy \eqref{171129@eq1a} if $\operatorname{div}u=0$ in $\Omega_R(x_0)$ and
$$
\int_{\Omega_R(x_0)}A^{\alpha\beta}D_\beta u\cdot D_\alpha \phi\,dx+\int_{\Omega_R(x_0)}p \operatorname{div} \phi\,dx=-\int_{\Omega_R(x_0)}f\cdot \phi\,dx
$$
for any $\phi\in W^1_2(\Omega_R(x_0))^d$ such that $\phi=0$ on $\partial B_R(x_0)\cap \Omega$.
\end{remark}

\begin{theorem}		\label{T2}
Let $d\ge 3$ and  $\Omega$ be a bounded domain in $\bR^d$ with
$|\Omega|\ge m_0>0$.
Let $(G, \Pi)$ be the Green function constructed in Theorem \ref{T1} under Assumptions \ref{A0} and \ref{A1}.
If we assume Assumption \ref{A2} (in addition to Assumptions \ref{A0} and \ref{A1}), then we have
\begin{equation}		\label{180115@A1}
|G(x,y)| \le C|x-y|^{2-d}\quad  \text{for any $x,y\in \Omega$ with $0<|x-y|<R_0$},
\end{equation}
where $C=C(d,\lambda,m_0,K_0,A_0,A_1)$.
If we further assume that there exists a constant $\theta>0$ satisfying
\begin{equation}		\label{170709@eq1}
|B_R(x_0)\setminus \Omega|\ge \theta R^d, \quad \forall x_0\in \partial \Omega, \quad \forall R\in (0, R_0),
\end{equation}
then for any $y\in \Omega$ and $0<R< R_0$, the following estimates hold:
\begin{enumerate}[$i)$]
\item
$\|G(\cdot,y)\|_{L_{2d/(d-2)}(\Omega\setminus B_R(y))}+\|DG(\cdot,y)\|_{L_2(\Omega\setminus B_R(y))}\le CR^{(2-d)/2}$.
\item
$\|\Pi(\cdot,y)\|_{L_2(\Omega\setminus B_R(y))}\le CR^{(2-d)/2}$.
\item
$\big|\{x\in \Omega:|G(x,y)|>t\}\big|\le C t^{-d/(d-2)}$ for all $t>  R_0^{2-d}$.
\item
$\big|\{x\in \Omega:|D_xG(x,y)|>t\}\big|\le C t^{-d/(d-1)}$ for all $t> R_0^{1-d}$.
\item
$\big|\{x\in \Omega:|\Pi(x,y)|>t\}\big|\le C t^{-d/(d-1)}$ for all $t > R_0^{1-d}$.
\item
$\|G(\cdot,y)\|_{L_q(\Omega_R(y))}\le C_qR^{2-d+d/q}$, where $q\in [1,d/(d-2))$.
\item
$\|DG(\cdot,y)\|_{L_q(\Omega_R(y))}\le C_q R^{1-d+d/q}$, where $q\in [1,d/(d-1))$.
\item
$\|\Pi(\cdot,y)\|_{L_q(\Omega_R(y))}\le C_q R^{1-d+d/q}$, where $q\in [1,d/(d-1))$.
\end{enumerate}
In the above, the constant $C$ depends only on $d$, $\lambda$, $m_0$, $K_0$, $A_0$, $A_1$, and $\theta$, and $C_q$ depends also on $q$.
\end{theorem}

\begin{remark}
If $|x-y|\ge R_0$, then $G(x,y)$ is bounded    by a constant depending only on $R_0$ and the parameters for the constant $C$ in \eqref{180115@A1}.
See \eqref{180202@eq2} for more details.
Thus in \eqref{180115@A1}, one can remove the condition $|x-y|<R_0$ if we allow the constant $C$ to depend also on $R_0$ and $\operatorname{diam}\Omega$.
\end{remark}

\section{Auxiliary results}		\label{S3}

In this section, we derive some auxiliary results which will be used in the proofs of the main theorems.
We do not impose any regularity assumption on the coefficients $A^{\alpha\beta}$ of the operator $\cL$.

\begin{lemma}		\label{170602@lem1}
Let $\Omega$ be a bounded domain in $\bR^d$ and $g\in L_{2}(\Omega)$.
Then there exists a function $u\in \tilde{W}^1_{2}(\Omega)^d$ such that
$$
\operatorname{div}u=g \quad \text{in }\, \Omega, \quad \|Du\|_{L_{2}(\Omega)}\le C\|g\|_{L_{2}(\Omega)},
$$
where the constant $C$ depends only on $d$.
\end{lemma}

\begin{proof}
See \cite[Lemma 3.1]{MR3809039}.
\end{proof}

The following lemma is regarding the $W^{1,2}$-estimate and solvability for the Stokes system.

\begin{lemma}		\label{170602@lem2}
Let $d\ge 3$ and  $\Omega$ be a bounded domain in $\bR^d$.
Then under Assumption \ref{A0},
for any $f\in \tilde{L}_{2d/(d+2)}(\Omega)^d$, $ f_\alpha\in L_2(\Omega)^d$, and $g\in L_2(\Omega)$,  there exists a unique $(u,p)\in \tilde{W}^1_2(\Omega)^d\times {L}_2(\Omega)$ satisfying
$$
\left\{
\begin{aligned}
\operatorname{div} u=g \quad &\text{in }\ \Omega,\\
\cL u+\nabla p=f+D_\alpha  f_\alpha \quad &\text{in }\ \Omega,\\
\cB u+p\nu = f_\alpha \nu_\alpha \quad &\text{on }\ \partial \Omega.
\end{aligned}
\right.
$$
Moreover, we have
$$
\|Du\|_{L_2(\Omega)}+\|p\|_{L_2(\Omega)}\le C\|f\|_{L_{2d/(d+2)}(\Omega)}+\tilde{C}\left(\|f_\alpha\|_{L_2(\Omega)}+\|g\|_{L_2(\Omega)}\right),
$$
where $C=C(d,\lambda,K_0)$ and $\tilde{C}=\tilde{C}(d,\lambda)$.
\end{lemma}

\begin{proof}
Notice from \eqref{170602@eq2} that $\tilde{W}^1_2(\Omega)$ is a Hilbert space with the inner product
$$
\langle u,v\rangle =\int_\Omega D_\alpha u\cdot D_\alpha v\,dx.
$$
Then the proof of the lemma is almost the same as that of \cite[Lemma 3.2]{MR3693868}.
We omit the details.
\end{proof}

We have the following Caccioppoli-type inequalities for the Stokes system.

\begin{lemma}		\label{171205@lem5}
\begin{enumerate}[$(a)$]
\item
Let $(u,p)\in W^1_2(B_R(x_0))^d\times L_2(B_R(x_0))$ satisfy
$$
\left\{
\begin{aligned}
\operatorname{div} u=0 \quad &\text{in }\ B_R(x_0),\\
\cL u+\nabla p=f \quad &\text{in }\ B_R(x_0),
\end{aligned}
\right.
$$
where $x_0\in \bR^d$, $R>0$, and $f\in L_\infty(B_R(x_0))^d$.
Then we have
$$
\begin{aligned}
&\|Du\|_{L_2(B_{R/2}(x_0))}+\|p-(p)_{B_{R/2}(x_0)}\|_{L_2(B_{R/2}(x_0))}\\
&\le C\big(R^{-1}\|u\|_{L_2(B_R(x_0))}+R^{(d+2)/2}\|f\|_{L_\infty(B_R(x_0))}\big),
\end{aligned}
$$
where $C=C(d,\lambda)$.
\item
Let $\Omega$ be a bounded domain in $\bR^d$.
Assume that there exist constants $\theta>0$ and $0<R_0\le 1$ satisfying
$$
|B_r(y_0)\setminus \Omega|\ge \theta r^d, \quad \forall y_0\in \partial \Omega, \quad \forall r\in (0, R_0).
$$
Let $(u,p)\in W^1_2(\Omega_R(x_0))^d\times L_2(\Omega_R(x_0))$ satisfy
\begin{equation}		\label{171205@B1}	
\left\{
\begin{aligned}
\operatorname{div} u=0 \quad &\text{in }\ \Omega_R(x_0),\\
\cL u+\nabla p=f \quad &\text{in }\ \Omega_R(x_0),\\
\cB u+p \nu=0 \quad &\text{on }\, \partial \Omega\cap B_R(x_0),
\end{aligned}
\right.
\end{equation}
where $x_0\in \partial{\Omega}$, $0<R<R_0$, and $f\in L_\infty(\Omega_R(x_0))^d$.
Then
we have
\begin{equation}	\label{171205@B2}	
\begin{aligned}
&\|Du\|_{L_2(\Omega_{R/2}(x_0))}+\|{p}\|_{L_2(\Omega_{R/2}(x_0))}\\
&\le C\big(R^{-1}\|u\|_{L_2(\Omega_R(x_0))}+R^{(d+2)/2}\|f\|_{L_\infty(\Omega_R(x_0))}\big),
\end{aligned}
\end{equation}
where $C=C(d,\lambda,\theta)$.
\end{enumerate}
\end{lemma}

\begin{proof}
The proof of the lemma proceeds in a standard manner.
See \cite{MR0641818} for Caccioppoli inequalities for both linear and nonlinear Stokes systems under some technical assumptions in a ball and a half ball.
For the reader's convenience, we present here the details of the proof of the assertion $(b)$, where we have a ball intersected with a domain satisfying an exterior measure condition.

Let $(u,p)\in W^1_2(\Omega_R(x_0))^d\times L_2(\Omega_R(x_0))$ satisfy \eqref{171205@B1}, where $x_0\in \partial \Omega$, $0<R<R_0$, and $f\in L_\infty(\Omega_R(x_0))^d$.
Let $r\in (0, R]$ be given.
We extend $p$ to $B_r(x_0)$ so that $(p)_{B_r(x_0)}=0$ and $\|p\|_{L_2(B_r(x_0))}$ is comparable to $\|p\|_{L_2(\Omega_r(x_0))}$.
By the existence of solutions to the divergence equation (with the homogeneous Dirichlet boundary condition and the right-hand side having zero mean) in a ball, there exists $\phi\in \mathring{W}^1_2(B_r(x_0))^d$ satisfying
$$
\operatorname{div}\phi=p \quad \text{in }\, B_r(x_0), \quad
\|D\phi\|_{L_2(B_r(x_0))}\le C\|p\|_{L_2(\Omega_r(x_0))},
$$
where $C=C(d, \theta)$.
We extend $\phi$ by zero on $\Omega\setminus B_r(x_0)$ and apply it as a test function to \eqref{171205@B1} to get
\begin{equation}		\label{171205@C1}
\|p\|_{L_2(\Omega_r(x_0))}\le C\big(\|Du\|_{L_2(\Omega_r(x_0))}+r^{(d+2)/2}\|f\|_{L_\infty(\Omega_r(x_0))}\big),
\end{equation}
where $C=C(d,\lambda, \theta)$.

Let $0<\rho<r\le R$ and $\eta$ be a smooth function on $\bR^d$ satisfying
$$
0\le \eta\le 1, \quad \eta\equiv 1 \, \text{ on }\, B_\rho(x_0), \quad \operatorname{supp}\eta \subset B_r(x_0), \quad |\nabla \eta|\le C(r-\rho)^{-1}.
$$
Applying $\eta^2 u$ as a test function to \eqref{171205@B1} and using \eqref{171205@C1}, we have
\begin{equation}		\label{171205@C2}
\begin{aligned}
\|Du\|_{L_2(\Omega_\rho(x_0))}&\le \varepsilon\|Du\|_{L_2(\Omega_r(x_0))}\\
&\quad+C\big((r-\rho)^{-1}\|u\|_{L_2(\Omega_r(x_0))}+ r^{(d+2)/2}\|f\|_{L_\infty(\Omega_r(x_0))}\big)
\end{aligned}
\end{equation}
for $\varepsilon\in (0,1)$, where $C=C(d,\lambda,\theta,\varepsilon)$.
For $k\in \{0,1,2,\ldots\}$, we set
$$
\varepsilon=\frac{1}{8}, \quad r_k=\frac{R}{2}\left(2-\frac{1}{2^k}\right)
$$
so that \eqref{171205@C2} becomes
\begin{align*}
\|Du\|_{L_2(\Omega_{r_k}(x_0))}&\le \varepsilon\|Du\|_{L_2(\Omega_{r_{k+1}}(x_0))}\\
&\quad  +C\big(2^k R^{-1}\|u\|_{L_2(\Omega_{R}(x_0))}+R^{(d+2)/2}\|f\|_{L_\infty(\Omega_{R}(x_0))}\big).
\end{align*}
Multiplying $\varepsilon^k$ and summing the estimates, we obtain that
$$
\|Du\|_{L_2(\Omega_{R/2}(x_0))}\le C\big(R^{-1} \|u\|_{L_2(\Omega_{R}(x_0))}+R^{(d+2)/2}\|f\|_{L_\infty(\Omega_{R}(x_0))}\big).
$$
Therefore, we get the desired estimate \eqref{171205@B2} from \eqref{171205@C1} and the above inequality.
The lemma is proved.
\end{proof}

\section{Approximated Green functions}		\label{S4}
Throughout this section, we assume that the hypotheses of Theorem \ref{T1} hold.
Under the hypotheses, we shall construct an approximated Green function and derive its various interior estimates.
We mainly follow the arguments in  Hofmann-Kim \cite{MR2341783}.

Let $y\in \Omega$, $\varepsilon\in (0,1]$, and $k\in \{1,\ldots,d\}$.
We denote by $(v,\pi)=(v_{\varepsilon,y,k}, \pi_{\varepsilon,y,k})$ the solution in $\tilde{W}^1_2(\Omega)^d\times L_2(\Omega)$ of the problem
\begin{equation}		\label{170609@eq2}
\left\{
\begin{aligned}
\operatorname{div} v=0 \quad &\text{in }\ \Omega,\\
\cL v+\nabla \pi=\Phi_{\varepsilon,y} e_k\quad &\text{in }\ \Omega,\\
\cB v+\pi\nu = 0 \quad &\text{on }\ \partial \Omega,
\end{aligned}
\right.
\end{equation}
where $e_k$ is the $k$-th unit vector in $\bR^d$ and
\begin{equation}		\label{170704@eq1}
\Phi_{\varepsilon,y}:=-\frac{1}{|\Omega_\varepsilon(y)|}I_{\Omega_\varepsilon(y)}+\frac{1}{|\Omega|}\in \tilde{L}_\infty(\Omega).
\end{equation}
Here, $I_{\Omega_{\varepsilon}(y)}$  is the characteristic function.
Note that, for each $\varepsilon \in (0,1]$ and $B_R(x) \subset \Omega$ with $R< R_0$,
$(v,\pi)$ satisfies
$$	
\left\{
\begin{aligned}
\operatorname{div} v=0 &\quad \text{in }\, B_R(x),\\
\cL v+\nabla \pi=\Phi_{\varepsilon, y}e_k &\quad \text{in }\,B_R(x),
\end{aligned}
\right.
$$
where $\Phi_{\varepsilon,y}e_k$ is a bounded function.
Thus, by Assumption \ref{A1} there exists a version of $v$ in $B_R(x)$ which is continuous on $\overline{B_{R/2}(x)}$.
Then there is a version $\tilde{v}$ of $v$ such that $\tilde{v} = v$ a.e. in $\Omega$ and $\tilde{v}$ is continuous in $\Omega$.
We define the approximated Green function $(G_\varepsilon(\cdot,y),\Pi_\varepsilon(\cdot,y))$ for $\cL$ by
$$
G^{jk}_{\varepsilon}(\cdot,y)= \tilde{v}^j=\tilde{v}^j_{\varepsilon,y,k}  \quad \text{and}\quad \Pi^k_\varepsilon(\cdot,y)=\pi=\pi_{\varepsilon,y,k}.
$$
Here, $G_{\varepsilon}(\cdot,y)$ is a $d\times d$ matrix-valued function and $\Pi_{\varepsilon}(\cdot,y)$ is a $1\times d$ vector-valued function.
By Lemma \ref{170602@lem2}, we have
\begin{align}		
\nonumber
\|DG_\varepsilon(\cdot,y)\|_{L_2(\Omega)}+\|\Pi_\varepsilon(\cdot,y)\|_{L_2(\Omega)} &\le C\big(|\Omega_\varepsilon(y)|^{(2-d)/(2d)}+|\Omega|^{(2-d)/(2d)}\big)\\
\label{170602@eq3}
& \le C |\Omega_\varepsilon(y)|^{(2-d)/(2d)},
\end{align}
where $C=C(d,\lambda,K_0)$.

In the lemma below, we obtain the pointwise bound for the approximated Green function.

\begin{lemma}		\label{170607@lem1}
For any $x,y\in \Omega$ satisfying
$$
0<2\varepsilon<\frac{|x-y|}{2}<\min\left\{R_0,\frac{1}{3}\operatorname{dist}(y,\partial\Omega)\right\},
$$
we have
$$
|G_\varepsilon(x,y)|\le C|x-y|^{2-d},
$$
where $C=C(d,\lambda,K_0,A_0)$.
\end{lemma}

\begin{proof}
Let $x,y\in \Omega$ and $k\in \{1,\ldots,d\}$.
Denote
\begin{equation}		\label{170714_02}
(v,\pi)=(G_\varepsilon^{\cdot k}(\cdot,y), \Pi^k_\varepsilon(\cdot,y)),
\end{equation}
where $G_{\varepsilon}^{\cdot k}(\cdot,y)$ is the $k$-th column of $G_\varepsilon(\cdot,y)$.
We assume that $0<2\varepsilon<R<\min\{R_0,\operatorname{dist}(y,\partial \Omega)\}$.
Find $(u,p)\in \tilde{W}^1_2(\Omega)^d\times L_2(\Omega)$ satisfying
$$	
\left\{
\begin{aligned}
\operatorname{div} u=0 &\quad \text{in }\, \Omega,\\
\cL^* u+\nabla p=-f +(f)_\Omega&\quad \text{in }\,\Omega,\\
\cB^* u+p\nu=0 &\quad \text{on }\, \partial \Omega,
\end{aligned}
\right.
$$
where $f=I_{\Omega_R(x)}\big(\operatorname{sgn}v^1, \ldots,\operatorname{sgn}v^d\big)^{\top}$. Then by Assumptions \ref{A0} and \ref{A1}, we have
\begin{align}		
\nonumber
\|u\|_{L_\infty(B_{R/2}(y))}&\le C\big(R^{1-d/2}\|u\|_{L_{2d/(d-2)}(B_R(y))}+R^2\big)\\
\nonumber
&\le C\big(R^{1-d/2}\|Du\|_{L_2(\Omega)}+R^2\big)\\
\label{171129@eq2}
&\le  C R^2,
\end{align}
where the last inequality is due to Lemma \ref{170602@lem2} and $C=C(d,\lambda,K_0,A_0)$.
Since we have
$$
\int_{\Omega_R(x)}f\cdot v\,dz=\int_\Omega A^{\alpha\beta}D_\beta v\cdot D_\alpha u\,dz=\dashint_{B_{\varepsilon}(y)}u^k\,dz,
$$
by \eqref{171129@eq2} and the fact that $\varepsilon<R/2$, we get
\begin{equation}		\label{170607@eq8}
\|v\|_{L_1(\Omega_R(x))}\le CR^2
\end{equation}
for any $x\in \Omega$ and $0<2\varepsilon<R<\min\{R_0,\operatorname{dist}(y,\partial \Omega)\}$.

Assume
$$
0<2\varepsilon<R:=\frac{|x-y|}{2}<\min\left\{R_0,\frac{1}{3}\operatorname{dist}(y,\partial\Omega)\right\}.
$$
Since $B_R(x)\subset \Omega$ and $B_R(x)\cap B_\varepsilon(y)=\emptyset$, we obtain by \eqref{170609@eq2} that
\begin{equation}		\label{180608@eq3a}
\left\{
\begin{aligned}
\operatorname{div} v=0 &\quad \text{in }\, B_R(x),\\
\cL v+\nabla \pi=\frac{e_k}{|\Omega|} &\quad \text{in }\,B_R(x).
\end{aligned}
\right.
\end{equation}
For any $z\in B_R(x)$ and $r\in (0, R]$ with $B_r(z)\subset B_R(x)$, by Assumption \ref{A1} applied to \eqref{180608@eq3a} in $B_r(z)$, we have
\begin{align*}
\|v\|_{L_\infty(B_{r/2}(z))}
&\le A_0\bigg(r^{-d/2}\|v\|_{L_2(B_r(z))}+\frac{r^2}{|\Omega|}\bigg)\\
&\le C\big(r^{-d/2}\|v\|_{L_2(B_r(z))}+R^{2-d}\big),
\end{align*}
where we used the fact that $r\le R$ and $|\Omega|\ge CR^d$ in the second inequality.
Hence by setting $w=|v|+R^{2-d}$, we see that
$$
\|w\|_{L_\infty(B_{r/2}(z))}\le C r^{-d/2}\|w\|_{L_2(B_r(z))}.
$$
Because the above inequality holds for all $z\in B_R(x)$ and $r\in (0, R]$ with $B_r(z)\subset B_R(x)$,
by a well known argument (see \cite[pp. 80--82]{MR1239172})
we have
$$
\|w\|_{L_\infty(B_{R/2}(x))}\le C R^{-d}\|w\|_{L_1(B_R(x))},
$$
which implies
\begin{equation}		\label{180621@eq1}
\|v\|_{L_\infty (B_{R/2}(x))}\le C\big(R^{-d}\|v\|_{L_1(B_R(x))}+R^{2-d}\big).
\end{equation}
Thus,  using \eqref{170607@eq8} and the continuity of $v$, we get
$$
|v(x)|\le CR^{2-d},
$$
where $C=C(d,\lambda,K_0,A_0)$.
The lemma is proved.
\end{proof}

\begin{lemma}		\label{170625@lem1}
Let $y\in \Omega$, $0<R< \min\big\{R_0,\frac{4}{5}\operatorname{dist}(y,\partial \Omega)\big\}$, and $0<\varepsilon<R/4$.
For $k\in \{1,\ldots,d\}$, we set
$$
\tilde{\Pi}^k_\varepsilon(\cdot,y)=\Pi^k_\varepsilon(\cdot,y)-(\Pi^k_\varepsilon(\cdot,y))_{B_R(y)\setminus B_{R/2}(y)}.
$$
Then we have
$$
\|\tilde{\Pi}^k_\varepsilon(\cdot,y)\|_{L_2(B_R(y)\setminus B_{R/2}(y))}\le \frac{C}{R}\|G_\varepsilon^{\cdot k}(\cdot,y)\|_{L_2(B_{5R/4}(y)\setminus B_{R/4}(y))}+CR^{(2-d)/2},
$$
where $G_{\varepsilon}^{\cdot k}(\cdot,y)$ is the $k$-th column of $G_\varepsilon(\cdot,y)$ and $C=C(d,\lambda)$.
\end{lemma}

\begin{proof}
Recall the notation \eqref{170714_02}, and set
$$
\tilde{\pi}=\pi-(\pi)_{B_R(y)\setminus B_{R/2}(y)}.
$$
Since $(\tilde{\pi})_{B_R(y)\setminus B_{R/2}(y)}=0$, by the existence of solutions to the divergence equation (see, for instance, \cite{MR2263708}), there exists a function $\phi\in \mathring{W}^1_2(B_R(y)\setminus \overline{B_{R/2}(y)})^d$ such that
$$
\operatorname{div}\phi=\tilde{\pi} \quad \text{in }\, B_R(y)\setminus \overline{B_{R/2}(y)}
$$
and
\begin{equation}		\label{170617@eq3}
\|D\phi\|_{L_2(B_R(y)\setminus B_{R/2}(y))}\le C\|\tilde{\pi}\|_{L_2(B_R(y)\setminus B_{R/2}(y))}.
\end{equation}
Here, by a scaling argument, one can  check that the constant $C$ in the above inequality  depends only on $d$.
We extend $\phi$ by zero on $\bR^d\setminus (B_R(y)\setminus \overline{B_{R/2}(y)})$ and apply $\phi$ as a test function to \eqref{170609@eq2} to get
\begin{align}
\nonumber
\int_{B_R(y)\setminus B_{R/2}(y)}|\tilde{\pi}|^2\,dx&=\int_{B_R(y)\setminus B_{R/2}(y)}\pi \tilde{\pi}\,dx\\
\label{170618_eq1}
&=-\int_{\Omega}A^{\alpha\beta}D_\beta v\cdot D_\alpha \phi\,dx-\dashint_\Omega \phi^k\,dx.
\end{align}
By H\"older's inequality, the Sobolev inequality, and \eqref{170617@eq3},  we have
\begin{align*}
\left|\dashint_\Omega \phi^k\,dx\right|&\le C|\Omega|^{-1}R^{1+d/2}\|\phi\|_{L_{2d/(d-2)}(B_R(y)\setminus B_{R/2}(y))}\\
&\le C|\Omega|^{-1}R^{1+d/2}\|D\phi\|_{L_2(B_R(y)\setminus B_{R/2}(y))}\\
&\le C|\Omega|^{-1}R^{1+d/2}\|\tilde{\pi}\|_{L_2(B_R(y)\setminus B_{R/2}(y))},
\end{align*}
where $C=C(d)$.
Combining these together, and using H\"older's inequality and Young's inequality, we see that
\begin{equation}		\label{170619@eq1}
\int_{B_R(y)\setminus B_{R/2}(y)}|\tilde{\pi}|^2\,dx\le C\int_{B_R(y)\setminus B_{R/2}(y)}|Dv|^2\,dx+C|\Omega|^{-2}R^{2+d},
\end{equation}
where $C=C(d,\lambda)$.

Let $z\in B_R(y)\setminus B_{R/2}(y)$.
Since $B_{R/4}(z)\cap B_\varepsilon(y)=\emptyset$, it follows from \eqref{170609@eq2} that
$$
\left\{
\begin{aligned}
\operatorname{div} v=0 &\quad \text{in }\, B_{R/4}(z),\\
\cL v+\nabla \pi=\frac{e_k}{|\Omega|} &\quad \text{in }\, B_{R/4}(z).
\end{aligned}
\right.
$$
Then by Lemma \ref{171205@lem5} $(a)$ applied to the above system and the fact that $B_{R/4}(z)\subset \big(B_{5R/4}(y)\setminus B_{R/4}(y)\big)$, we have
\begin{equation}		\label{180621@eq4}
\int_{B_{R/8}(z)}|Dv|^2\,dx\le \frac{C}{R^2}\int_{B_{5R/4}(y)\setminus B_{R/4}(y)}|v|^2\,dx+C|\Omega|^{-2}R^{2+d},
\end{equation}
where $C=C(d,\lambda)$.
Because the above inequality holds for any $z\in B_{R}(y)\setminus B_{R/2}(y)$,
by taking points $z^1,\ldots,z^n$, where $n=n(d)$, in $B_R(y)\setminus B_{R/2}(y)$ such that
$$
(B_R(y)\setminus B_{R/2}(y) ) \subset \bigcup_{i=1}^n B_{R/8}(z^i),
$$
and using \eqref{180621@eq4} with $B_{R/8}(z^i)$ in place of $B_{R/8}(z)$, we have
\begin{align}
\nonumber
\int_{B_{R}(y)\setminus B_{R/2}(y)}|Dv|^2\,dx
&\le \sum_{i=1}^n \int_{B_{R/8}(z^i)}|Dv|^2\,dx\\
\label{180621@eq4a}
&\le \frac{C}{R^2}\int_{B_{5R/4}(y)\setminus B_{R/4}(y)} |v|^2\,dx+C|\Omega|^{-2}R^{2+d}.
\end{align}
Combining \eqref{170619@eq1} and \eqref{180621@eq4a}, and using $|\Omega|\ge CR^d$, we get the desired estimate.
\end{proof}

Based on Lemmas \ref{170607@lem1} and \ref{170625@lem1}, we obtain the following uniform estimates for $(G_\varepsilon(\cdot,y), \Pi_\varepsilon(\cdot,y))$ away from the pole $y$.

\begin{lemma}		\label{170710@lem1}
Let $y\in \Omega$, $0<R<\min\{R_0,\operatorname{dist}(y,\partial \Omega)\}$, and $0<\varepsilon\le 1$.
Then we have
\begin{equation}		\label{170626@eq1}
\|G_\varepsilon(\cdot,y)\|_{L_{2d/(d-2)}(\Omega\setminus B_R(y))}+\|DG_\varepsilon(\cdot,y)\|_{L_2(\Omega\setminus B_R(y))}\le CR^{(2-d)/2},
\end{equation}
\begin{equation}		\label{170626@eq2}
\|\Pi_\varepsilon(\cdot,y)\|_{L_2(\Omega\setminus B_R(y))}\le CR^{(2-d)/2},
\end{equation}
where $C=C(d,\lambda,K_0,A_0)$.
\end{lemma}

\begin{proof}
Fix $y\in \Omega$ and $k\in \{1,\ldots,d\}$.
Recall the notation \eqref{170714_02}.
We first prove that the estimate \eqref{170626@eq1} holds.
Certainly, we may assume that  $0<R< \frac{1}{3}\operatorname{dist}(y,\partial \Omega)$.
For $R/16\le \varepsilon\le 1$, we obtain by \eqref{170602@eq2} and \eqref{170602@eq3} that
$$
\|v\|_{L_{2d/(d-2)}(\Omega)}+\|Dv\|_{L_2(\Omega)}\le C|B_{R/16}(y)|^{(2-d)/(2d)}\le CR^{(2-d)/2},
$$
which gives \eqref{170626@eq1}.
Assume $0<\varepsilon<R/16$, and let $\eta$ be an infinitely differentiable function on $\bR^d$ satisfying
\begin{equation}		\label{170628@eq1}
0\le \eta\le 1, \quad \eta\equiv 1 \, \text{ on }\, B_{R/2}(y), \quad \operatorname{supp}\eta\subset B_R(y), \quad |\nabla\eta|\le CR^{-1}.
\end{equation}
Applying $(1-\eta)^2v$ as a test function to \eqref{170609@eq2}, we have
\begin{equation}		\label{171129@eq3}
\|(1-\eta)Dv\|_{L_2(\Omega)}^2 \le \frac{C}{R^2}\|v\|^2_{L_2(B_R(y)\setminus B_{R/2}(y))}+ C(I_1+I_2),
\end{equation}
where
$$
I_1=\left|\int_\Omega \pi\operatorname{div}((1-\eta)^2v)\,dx\right|, \quad I_2=\left|\dashint_\Omega (1-\eta)^2v^k\,dx\right|.
$$
Notice from  $(v)_\Omega=0$ and \eqref{170607@eq8} that
$$
I_2=\left|\dashint_\Omega (\eta^2-2\eta)v\,dx\right|\le \frac{C}{|\Omega|}\|v\|_{L_1(B_R(y))}\le CR^{2-d},
$$
where $C=C(d,\lambda,K_0,A_0)$.
To estimate $I_1$, we use $\operatorname{div} v=0$ to get
\begin{align*}
I_1&=\left|\int_\Omega \pi\operatorname{div}\big((\eta^2-2\eta)v\big)\,dx\right|=\bigg|\int_{B_R(y)\setminus B_{R/2}(y)}\tilde{\pi}\nabla (\eta^2-2\eta)\cdot v\,dx\bigg|,
\end{align*}
where $\tilde{\pi}=\pi-(\pi)_{B_R(y)\setminus B_{R/2}(y)}$.
Then by H\"older's inequality and Lemma \ref{170625@lem1}, we find
$$
I_1\le \frac{C}{R^2}\|v\|_{L_2(B_{5R/4}(y)\setminus B_{R/4}(y))}^2+CR^{2-d},
$$
where $C=C(d,\lambda)$.
Combining \eqref{171129@eq3} and the estimates of $I_1$ and $I_2$,
we have
\begin{equation}		\label{171129@eq3a}
\|(1-\eta)Dv\|_{L_2(\Omega)} \le \frac{C}{R}\|v\|_{L_2(B_{5R/4}(y)\setminus B_{R/4}(y))}+CR^{(2-d)/2},
\end{equation}
where $C=C(d,\lambda,K_0,A_0)$.
Notice from $(v)_\Omega=0$ and \eqref{170602@eq2} that
\begin{align}
\nonumber
&\|(1-\eta)v\|_{L_{2d/(d-2)}(\Omega)}\\
\nonumber
&\le \big\|(1-\eta)v-((1-\eta)v)_\Omega\|_{L_{2d/(d-2)}(\Omega)}+|\Omega|^{\frac{d-2}{2d}}\big|((1-\eta)v)_\Omega\big|\\
\nonumber
&= \big\|(1-\eta)v-((1-\eta)v)_\Omega\|_{L_{2d/(d-2)}(\Omega)}+|\Omega|^{\frac{d-2}{2d}}\big|(\eta v)_\Omega\big|\\
\label{171201@eq3a}
&\le C\|(1-\eta)Dv\|_{L_2(\Omega)}+\frac{C}{R}\|v\|_{L_2(B_R(y)\setminus B_{R/2}(y))}+CR^{-(2+d)/2}\| v\|_{L_{1}(B_R(y))},
\end{align}
where $C=C(d,K_0)$.
This together with \eqref{171129@eq3a} and \eqref{170607@eq8} yields that
\begin{equation}		\label{171129@eq4}
\begin{aligned}
&\|(1-\eta)v\|_{L_{2d/(d-2)}(\Omega)}+\|(1-\eta)Dv\|_{L_2(\Omega)}\\
&\le \frac{C}{R}\|v\|_{L_2(B_{5R/4}(y)\setminus B_{R/4}(y))}+CR^{(2-d)/2}.
\end{aligned}
\end{equation}
Since we have
$$
2\varepsilon<\frac{R}{8}<\frac{|x-y|}{2}<\frac{5R}{8}<\min\left\{ R_0, \frac{1}{3}\operatorname{dist}(y,\partial \Omega)\right\}
$$
for all $x\in B_{5R/4}(y)\setminus B_{R/4}(y)$, we obtain by Lemma \ref{170607@lem1} that
$$
\frac{1}{R}\|v\|_{L_2(B_{5R/4}(y)\setminus B_{R/4}(y))}\le CR^{(2-d)/2}.
$$
Therefore, we get the estimate \eqref{170626@eq1} from \eqref{171129@eq4} and the above inequality.

We now turn to the estimate \eqref{170626@eq2}.
Similar to the above, it suffices to show the estimate with $0<R<\frac{1}{3}\operatorname{dist}(y,\partial \Omega)$ and $0<\varepsilon<R/16$.
By Lemma \ref{170602@lem1} and \eqref{170602@eq2}, there exists $\phi\in \tilde{W}^1_2(\Omega)^d$ such that
\begin{equation}		\label{170628@eq2a}
\operatorname{div}\phi=\pi I_{\Omega\setminus B_R(y)} \quad \text{in }\, \Omega
\end{equation}
and
\begin{equation}		\label{170628@eq4a}
\|\phi\|_{L_{2d/(d-2)}(\Omega)}+\|D\phi\|_{L_2(\Omega)}\le C(d,K_0)\|\pi\|_{L_2(\Omega\setminus B_R(y))}.
\end{equation}
We apply $(1-\eta)\phi$ as a test function to \eqref{170609@eq2}, where $\eta$ is as in \eqref{170628@eq1}, to get
\begin{equation}		\label{170628@eq2}
\begin{aligned}
&\int_\Omega \pi \operatorname{div}((1-\eta)\phi)\,dx\\
&=-\int_\Omega A^{\alpha\beta}D_\beta v\cdot D_\alpha ((1-\eta)\phi)\,dx-\dashint_\Omega (1-\eta)\phi^{k}\,dx.
\end{aligned}
\end{equation}
Using \eqref{170628@eq2a},
\begin{align*}
\int_\Omega \pi\operatorname{div}((1-\eta)\phi)\,dx&=\int_{\Omega}\pi \operatorname{div}\phi\,dx-\int_\Omega \pi\operatorname{div}(\eta \phi)\,dx\\
&=\int_{\Omega \setminus B_R(y)}|\pi|^2\,dx-\int_{B_R(y)\setminus B_{R/2}(y)}\tilde{\pi} \nabla\eta \cdot \phi\,dx,
\end{align*}
where $\tilde{\pi}=\pi-(\pi)_{B_R(y)\setminus B_{R/2}(y)}$.
Combining these together and using H\"older's inequality, we have
$$
\begin{aligned}
\|\pi\|_{L_2(\Omega\setminus B_R(y))}^2 &\le C\|Dv\|_{L_2(\Omega\setminus B_{R/2}(y))}\big(\|\phi\|_{L_{2d/(d-2)}(\Omega)}+\|D\phi\|_{L_2(\Omega)}\big)\\
&\quad +C\big(\|\tilde{\pi}\|_{L_2(B_R(y)\setminus B_{R/2}(y))}+|\Omega|^{(2-d)/2d}\big)\|\phi\|_{L_{2d/(d-2)}(\Omega)},
\end{aligned}
$$
and thus, by \eqref{170628@eq4a} and $|B_R|\le |\Omega|$, we obtain
\begin{equation}		\label{180625@eq4a}
\|\pi\|_{L_2(\Omega\setminus B_R(y))}\le C\big(\|Dv\|_{L_2(\Omega\setminus B_{R/2}(y))}+\|\tilde{\pi}\|_{L_2(B_R(y)\setminus B_{R/2}(y))}+R^{(2-d)/2}\big).
\end{equation}
This inequality together with Lemma \ref{170625@lem1} and \eqref{170626@eq1} implies \eqref{170626@eq2}.
The lemma is proved.
\end{proof}

From Lemma \ref{170710@lem1}, we get the following uniform weak type estimates.

\begin{lemma}		\label{171130@lem1}
Let $y\in \Omega$ and $0<\varepsilon\le 1$.
Then we have
\begin{align*}
\big|\{x\in \Omega:|G_\varepsilon(x,y)|>t\}\big|\le C t^{-d/(d-2)}, \quad \forall t> \min\{R_0,\operatorname{dist}(y,\partial \Omega)\}^{2-d},\\
\big|\{x\in \Omega:|D_xG_\varepsilon(x,y)|>t\}\big|\le C t^{-d/(d-1)}, \quad \forall t> \min\{R_0,\operatorname{dist}(y,\partial \Omega)\}^{1-d},\\
\big|\{x\in \Omega:|\Pi_\varepsilon(x,y)|>t\}\big|\le C t^{-d/(d-1)}, \quad \forall t> \min\{R_0,\operatorname{dist}(y,\partial \Omega)\}^{1-d},
\end{align*}
where $C=C(d,\lambda,K_0,A_0)$.
\end{lemma}

\begin{proof}
We only prove the last inequality because the others are the same with obvious modifications.
For $y\in \Omega$ and $t> \min\{R_0,\operatorname{dist}(y,\partial \Omega)\}^{1-d}$, we set
$$
A_t=\{x\in \Omega:|\Pi_\varepsilon(x,y)|>t\}, \quad R=t^{-1/(d-1)}<\min\{R_0,\operatorname{dist}(y,\partial \Omega)\}.
$$
Then by \eqref{170626@eq2}, we have
$$
|A_t\setminus B_R(y)|\le \frac{1}{t^2}\int_{A_t\setminus B_R(y)}|\Pi_\varepsilon(x,y)|^2\,dx\le Ct^{-d/(d-1)}.
$$
On the other hand, we have
$$
|A_t\cap B_R(y)|\le CR^d=Ct^{-d/(d-1)}.
$$
Combining the above inequalities, we get the desired estimate.
\end{proof}

The following uniform $L_q$-estimates are easy consequences of Lemma \ref{171130@lem1}.

\begin{lemma}		\label{170710@lem2}
Let $y\in \Omega$, $0<R<\min\{R_0,\operatorname{dist}(y,\partial \Omega)\}$, and $0<\varepsilon\le 1$.
Then we have
\begin{align*}
\|G_{\varepsilon}(\cdot,y)\|_{L_q(B_R(y))}\le CR^{2-d+d/q}, \quad q\in [1, d/(d-2)),\\
\|DG_{\varepsilon}(\cdot,y)\|_{L_q(B_R(y))}\le CR^{1-d+d/q}, \quad q\in [1,d/(d-1)),\\
\|\Pi_{\varepsilon}(\cdot,y)\|_{L_q(B_R(y))}\le CR^{1-d+d/q}, \quad q\in [1,d/(d-1)),
\end{align*}
where $C=C(d,\lambda,K_0,A_0,q)$.
\end{lemma}

\begin{proof}
We only prove the last inequality because the others are the same with obvious modifications.
Let $y\in \Omega$, $0<R<\min\{R_0,\operatorname{dist}(y,\partial \Omega)\}$, and $q\in [1,d/(d-1))$, and set
$$
t=R^{1-d}, \quad A_t=\{x\in \Omega:|\Pi_\varepsilon(x,y)|>t\}.
$$
Then we have
\begin{align*}
\int_{B_R(y)}|\Pi_\varepsilon(x,y)|^q\,dx&=\int_{B_R(y)\setminus A_t}|\Pi_\varepsilon(x,y)|^q\,dx+\int_{B_R(y)\cap A_t}|\Pi_\varepsilon(x,y)|^q\,dx\\
&\le C R^{(1-d)q+d}+\int_{A_t}|\Pi_\varepsilon(x,y)|^q\,dx.
\end{align*}
Notice from the last inequality in Lemma \ref{171130@lem1} that
\begin{align*}
\int_{A_t}|\Pi_\varepsilon(x,y)|^q\,dx&=q\int_0^{\infty} s^{q-1}\big|\{x\in A_t:|\Pi_\varepsilon(x,y)|>s\}\big|\,ds\\
&\le q\int_0^t s^{q-1} |A_t|\,ds+q\int_t^\infty s^{q-1} |A_s|\,ds\\
&\le C R^{(1-d)q+d}.
\end{align*}
Combining the above inequalities, we get the desired estimate.
\end{proof}

\section{Proofs of main theorems}		\label{S5}
Throughout this section, for $y\in \Omega$ and $\varepsilon\in (0,1]$, we denote by $(G_\varepsilon(\cdot,y), \Pi_\varepsilon(\cdot,y))$ the approximated Green function constructed in Section \ref{S4}.

\begin{proof}[Proof of Theorem \ref{T1}]
The proof is a modification of the proof of  \cite[Theorem 4.1]{MR2341783}.
Let $y\in \Omega$ and  $\varepsilon\in (0,1]$.
By Lemmas \ref{170710@lem1} and \ref{170710@lem2}, there exist a sequence $\{\varepsilon_\rho\}_{\rho=1}^\infty$ tending to zero and a pair $(G(\cdot,y),\Pi(\cdot,y))$ such that for $0<R<\min\{R_0,\operatorname{dist}(y,\partial \Omega)\}$,
\begin{equation}		\label{170704@eq2}
\begin{aligned}
G_{\varepsilon_\rho}(\cdot,y) \rightharpoonup G(\cdot,y) &\quad \text{weakly in }\, W^1_2(\Omega\setminus B_R(y))^{d\times d},\\
\Pi_{\varepsilon_\rho}(\cdot,y) \rightharpoonup \Pi(\cdot,y) &\quad \text{weakly in }\, L_2(\Omega\setminus B_R(y))^{d},
\end{aligned}
\end{equation}
and
\begin{equation}		\label{170704@eq2a}
\begin{aligned}
G_{\varepsilon_\rho}(\cdot,y) \rightharpoonup G(\cdot,y) &\quad \text{weakly in }\, W^1_q(B_R(y))^{d\times d},\\
\Pi_{\varepsilon_\rho}(\cdot,y) \rightharpoonup \Pi(\cdot,y) &\quad \text{weakly in }\, L_q(B_R(y))^{d},
\end{aligned}
\end{equation}
where $1<q<d/(d-1)$.
Then one can easily check that $(G(\cdot,y),\Pi(\cdot,y))$ satisfies properties (a) -- (c) in Definition \ref{D1}, which means that $(G,\Pi)$ is the Green function of $\cL$ in the domain $\Omega$.
We note that by the property (b) and Assumption \ref{A1}, $G(\cdot,y)$ is continuous in $\Omega\setminus \{y\}$. More precisely, we choose a version of $G$ which is continuous in $\Omega \setminus \{y\}$ and denote it again by $G$.
We also remark that the identity \eqref{180123@eq1} holds for {\emph{all}} $y\in \Omega$ if $g\equiv 0$.
Indeed, if $(u,p)\in \tilde{W}^1_2(\Omega)^d\times L_2(\Omega)$ satisfies \eqref{170129@eq3} with $g\equiv 0$, then by Assumption \ref{A1} there is a continuous version of $u$ in $\Omega$, denoted again by $u$, satisfying
$$
\dashint_{\Omega_{\varepsilon_\rho}(y)} u\,dx=-\int_\Omega G_{\varepsilon_\rho}(x,y)^\top f(x)\,dx \quad \text{for all \,$y\in \Omega$}.
$$
By taking $\rho\to \infty$ and using the continuity of $u$, we have
\begin{equation}		\label{180129@eq5}
u(y)=-\int_\Omega G(x,y)^\top f(x)\,dx \quad \text{for all \,$y\in \Omega$}.
\end{equation}
The estimates $i)$ and $ii)$ in the theorem are easy consequences of Lemma \ref{170710@lem1} and \eqref{170704@eq2}.
Thus following the proofs of Lemmas \ref{171130@lem1} and \ref{170710@lem2}, we see that the estimates $iii)$ -- $viii)$ hold.
To show \eqref{170702@eq1}, let $x,y\in \Omega$ with $0<|x-y|<\frac{1}{2}\min\{R_0,\operatorname{dist}(y,\partial \Omega)\}$, and denote $r=|x-y|/2$.
Since $(G(\cdot,y), \Pi(\cdot,y))$ satisfies
$$
\left\{
\begin{aligned}
\operatorname{div} G(\cdot,y)=0 \quad &\text{in }\, B_r(x),\\
\cL G(\cdot,y)+\nabla \Pi(\cdot,y)=\frac{1}{|\Omega|}I \quad &\text{in }\, B_r(x),
\end{aligned}
\right.
$$
by Assumption \ref{A1}, H\"older's inequality, and the estimate $i)$ in the theorem, we have
\begin{align*}
|G(x,y)|
&\le C\left(r^{-d/2}\|G(\cdot,y)\|_{L_2(B_r(x))}+\frac{r^2}{|\Omega|}\right)\\
&\le C\big(r^{(2-d)/2}\|G(\cdot,y)\|_{L_{2d/(d-2)}(B_r(x))}+r^{2-d}\big)\\
&\le C\big(r^{(2-d)/2}\|G(\cdot,y)\|_{L_{2d/(d-2)}(\Omega\setminus B_r(y))}+r^{2-d}\big)\\
&\le C r^{2-d},
\end{align*}
where $C=C(d,\lambda,K_0,A_0)$.
This gives the pointwise bound \eqref{170702@eq1}.

For $x\in \Omega$ and $\sigma\in (0,1]$, let $(G_\sigma^*(\cdot,x),\Pi^*_\sigma(\cdot,x))$ be the approximated Green function for $\cL^*$, i.e., if we set $w=w_{\sigma,x,\ell}$ as the $\ell$-th column of $G_\sigma^*(\cdot,x)$ and $\kappa=\kappa_{\sigma,x,\ell}$ as the $\ell$-th component of $\Pi^*_\sigma(\cdot,x)$, then $(w,\kappa)\in \tilde{W}^1_2(\Omega)^d\times L_2(\Omega)$ satisfies
\begin{equation}		\label{171130@eq2}
\left\{
\begin{aligned}
\operatorname{div} w=0 \quad &\text{in }\ \Omega,\\
\cL^* w+\nabla \kappa=\Phi_{\sigma,x} e_\ell\quad &\text{in }\ \Omega,\\
\cB^* w+\kappa\nu = 0 \quad &\text{on }\ \partial \Omega,
\end{aligned}
\right.
\end{equation}
where $\Phi_{\sigma,x}$ is given in \eqref{170704@eq1}.
By proceeding similarly as above, there exist a sequence $\{\sigma_\tau\}_{\tau=1}^\infty$ tending to zero and the Green function $(G^*(\cdot,x),\Pi^*(\cdot,x))$ for $\cL^*$ such that $\left(G^*_{\sigma_\tau}(\cdot,x), \Pi^*_{\sigma_\tau}(\cdot,x)\right)$ and $\left(G^*(\cdot,x), \Pi^*(\cdot,x)\right)$ satisfy the natural counterparts of \eqref{170704@eq2}, \eqref{170704@eq2a}, and the properties of the Green function for $\cL$.
Notice from \eqref{180129@eq5} that
\begin{equation}		\label{180208@B1}
G^*_\sigma(y,x)=\dashint_{\Omega_\sigma(x)} G(z,y)^\top\,dz \quad \text{for all \,$x,y\in \Omega$}.
\end{equation}
Then by the continuity of $G(\cdot,y)$ on $\Omega\setminus \{y\}$, we have
\begin{equation}		\label{180208@A1}
\lim_{\sigma\to 0}G_{\sigma}^*(y,x)= G(x,y)^\top \quad \text{for all \,$ x,y\in \Omega$ \,with\, $x\neq y$}.
\end{equation}

Now we prove \eqref{170702@eq2}.
Let $y\in \Omega$ be given.
Then there exists a measure zero set $N_y\subset \Omega$ containing $y$ such that, by passing to a subsequence,
\begin{equation}		\label{180208@A1a}
\lim_{\rho\to \infty} G_{\varepsilon_\rho}(x,y)=G(x,y) \quad \text{for all }\, x\in \Omega\setminus N_y.
\end{equation}
Indeed, since it holds that
$$
\|G_{\varepsilon_\rho}(\cdot,y)\|_{W^1_1(\Omega)}\le C(d,\lambda,K_0,A_0, R_0, \operatorname{dist}(y,\partial \Omega)),
$$
by the Rellich-Kondrachov compactness theorem, for a sufficiently small $\delta>0$, there exists a subsequence of $\{G_{\varepsilon_\rho}(\cdot,y)\}$ which converges a.e. to $G(\cdot,y)$ on $\Omega_\delta$, where $\Omega_\delta$ is a smooth subdomain satisfying
$$
\{x\in \Omega:B_\delta(x)\subset \Omega\}\subset \Omega_\delta \subset \Omega.
$$
Thus by a diagonalization process, one can easily see that \eqref{180208@A1a} holds.
Combining \eqref{180208@A1a} and the counterpart of \eqref{180208@A1}, we have
$$
G(x,y)=G^*(y,x)^\top
$$
for all $x\in \Omega\setminus N_y$.
Similarly, we see that the above identity holds for all $y\in \Omega\setminus N_x$.
This gives \eqref{170702@eq2}.
An easy consequence of \eqref{170702@eq2} and the counterpart of \eqref{180208@B1} is that
$$
G_{\varepsilon}(x,y)=\dashint_{\Omega_\varepsilon(y)} G(x,z)\,dz \quad \text{for all \,$x,y\in \Omega$}.
$$

We finish the proof of Theorem \ref{T1} with the following remark.
One may prove the existence of the Green functions without the continuity condition in Assumption \ref{A1}.
In this case, the continuity of the Green functions is not guaranteed and the estimate \eqref{170702@eq1} should be replaced by
$$
\operatorname*{ess\,sup}_{B_{|x-y|/4}(x)}|G(\cdot ,y)|\le C|x-y|^{2-d}.
$$
Moreover, it is questionable whether one can prove \eqref{170702@eq2}.
Indeed, because of the lack of the continuity for solutions (containing the Green functions and the approximated Green functions), the identity in \eqref{180208@B1} holds for all $x\in \Omega$ and $y\in \Omega\setminus N_x$, where $N_x\subset \Omega$ is a measure zero set depending on $x$.
Thus if we take the limit as $\sigma\to 0$ in \eqref{180208@B1}, then the right-hand side is not well defined, which means the type of result in \eqref{180208@A1} is not available.
This makes it impossible to obtain \eqref{170702@eq2} by using our argument.
On the other hand, if one has a modulus of continuity estimate such as a H\"older estimate for the solution $u$ in Assumption \ref{A1}, then by
following the argument in \cite[Section 3.6]{MR2341783}, one can show that  \eqref{170702@eq2} holds for all $x,y\in \Omega$ with $x\neq y$.
\end{proof}

We now turn to Theorem \ref{T2}, which is about the global pointwise bound and various boundary estimates for the Green function.
To prove the theorem, we use the following boundary estimate.

\begin{lemma}		\label{170708@lem1}
Let $\Omega$ be a bounded domain in $\bR^d$ with $|\Omega|\ge m_0>0$.
Suppose that Assumptions \ref{A0} and  \ref{A1} hold, and  there exists a constant $\theta>0$ satisfying \eqref{170709@eq1}.
Let $y\in \Omega$, $y_0\in \partial \Omega$, $0<R< R_0$, and $0<\varepsilon\le 1$ with
$$
B_\varepsilon(y)\subset  B_{R/8}(y_0).
$$
Then we have
$$
\|{\Pi}_\varepsilon(\cdot,y)\big\|_{L_2(\Omega_{R/2}(y_0)\setminus B_{R/4}(y_0))}\le \frac{C}{R}\|G_\varepsilon(\cdot,y)\|_{L_2(\Omega_{5R/8}(y_0)\setminus B_{R/8}(y_0))}+C,
$$
where $C=C(d,\lambda,m_0,\theta)$.
\end{lemma}

\begin{proof}
Denote $(v,\pi)=(G_\varepsilon^{\cdot k}(\cdot,y),\Pi^k_\varepsilon(\cdot,y))$ and set
$$
\tilde{\pi}=\pi I_{\Omega}-(\pi I_{\Omega})_{B_{R/2}(y_0)\setminus B_{R/4}(y_0)}.
$$
By following the same argument in the proof of  \eqref{170619@eq1}, and using the fact that $|\Omega|\ge m_0$ and $R\le 1$, we have
\begin{equation}		\label{180625@A1}
\int_{B_{R/2}(y_0)\setminus B_{R/4}(y_0)} |\tilde{\pi}|^2\,dx\le C\int_{\Omega_{R/2}(y_0)\setminus B_{R/4}(y_0)}|Dv|^2\,dx+C,
\end{equation}
where $C=C(d,\lambda, m_0)$.
Moreover, by utilizing both $(a)$ and $(b)$ in Lemma \ref{171205@lem5}, and following the same steps used in deriving \eqref{180621@eq4a}, we obtain
\begin{equation}		\label{180625@A2}
\int_{\Omega_{R/2}(y_0)\setminus B_{R/4}(y_0)}|Dv|^2\,dx\le \frac{C}{R^2} \int_{\Omega_{5R/8}(y_0)\setminus B_{R/8}(y_0)}|v|^2\,dx+C,
\end{equation}
where $C=C(d,\lambda,m_0,\theta)$.

In the rest of the proof, we shall prove that
\begin{equation}		\label{180202@eq5a}
\begin{aligned}
\int_{\Omega_{R/2}(y_0)\setminus B_{R/4}(y_0)}|\pi|^2\,dx
&\le  C\int_{\Omega_{R/2}(y_0)\setminus B_{R/4}(y_0)}|Dv|^2\,dx\\
&\quad+\int_{B_{R/2}(y_0)\setminus B_{R/4}(y_0)}|\tilde{\pi}|^2\,dx+C,
\end{aligned}
\end{equation}
where $C=C(d,\lambda,m_0, \theta)$, which together with \eqref{180625@A1} and \eqref{180625@A2} implies the desired estimate.
Set
$$
\hat{\pi}=
\left\{
\begin{aligned}
0 & \quad \text{on }\, B_{R/4}(y_0),\\
\pi  &\quad \text{on }\, \Omega_{R/2}(y_0)\setminus B_{R/4}(y_0),\\
-\frac{1}{|\cD|}\int_{\Omega_{R/2}(y_0)\setminus B_{R/4}(y_0)}\pi\,dx &\quad \text{on }\, \cD,
\end{aligned}
\right.
$$
where $\cD=B_{R/2}(y_0)\setminus (\Omega\cup B_{R/4}(y_0))$.
Note that
\begin{equation}		\label{180608@eq1}
|\cD|\ge \theta (R/8)^d.
\end{equation}
Indeed, if we take a point $z_0\in \partial \Omega \cap \partial B_{3R/8}(y_0)$, then by \eqref{170709@eq1} and the fact that $(B_{R/8}(z_0)\setminus \Omega) \subset \cD$, we have
$$
|\cD|\ge |B_{R/8}(z_0)\setminus \Omega|\ge  \theta (R/8)^d.
$$
Since
$(\hat{\pi})_{B_{R/2}(y_0)}=0$ and \eqref{180608@eq1} yields
$$
\|\hat{\pi}\|_{L_2(B_{R/2}(y_0))}\le C\|\pi\|_{L_2(\Omega_{R/2}(y_0)\setminus B_{R/4}(y_0))},
$$
by the existence of solutions to the divergence equation in a ball,
there exists a function  $\phi\in \mathring{W}^1_2(B_{R/2}(y_0))^d$ such that
$$
\operatorname{div}\phi=\hat{\pi} \quad \text{in }\, B_{R/2}(y_0),
$$
\begin{equation}		\label{170711_eq1}
\|\phi\|_{L_{2d/(d-2)}(B_{R/2}(y_0))}+\|D\phi\|_{L_2(B_{R/2}(y_0))}\le C\|{\pi}\|_{L_2(\Omega_{R/2}(y_0)\setminus B_{R/4}(y_0))},
\end{equation}
where $C=C(d, \theta)$.
We extend $\phi$ to be zero on $\bR^d \setminus B_{R/2}(y_0)$.
Let $\eta$ be an infinitely differentiable function on $\bR^d$ satisfying
$$
0\le \eta\le 1, \quad \eta\equiv 1 \, \text{ on }\, B_{R/4}(y_0), \quad \operatorname{supp}\eta\subset B_{R/2}(y_0), \quad |\nabla \eta|\le CR^{-1}.
$$
We apply $(1-\eta)\phi$ as a test function to \eqref{170609@eq2} to get
\begin{equation}		\label{180627@eq1}
\int_\Omega \pi\operatorname{div}((1-\eta)\phi)\,dx=-\int_\Omega A^{\alpha\beta}D_\beta v\cdot D_\alpha ((1-\eta)\phi)\,dx-\dashint_\Omega (1-\eta)\phi^k\,dx.
\end{equation}
Observe that
\begin{align}
\nonumber
&\int_\Omega \pi\operatorname{div}((1-\eta)\phi)\,dx\\
\nonumber
&=\int_\Omega \pi\operatorname{div}\phi\,dx-\int_\Omega \pi\operatorname{div}(\eta\phi)\,dx\\
\label{180627@eq2}
&=\int_{\Omega_{R/2}(y_0)\setminus B_{R/4}(y_0)}|\pi|^2\,dx-\int_{B_{R/2}(y_0)\setminus B_{R/4}(y_0)} \tilde{\pi}\operatorname{div}(\eta \phi)\,dx.
\end{align}
Indeed, since
$$
\operatorname{div}(\eta\phi)=\nabla \eta\cdot \phi+\eta \operatorname{div}\phi=\nabla \eta \cdot \phi+\eta \hat{\pi} \quad \text{in }\, B_{R/2}(y_0),
$$
by the definitions of $\eta$ and $\hat{\pi}$, we have
$$
\operatorname{supp} (\operatorname{div}(\eta\phi))\subset \big(B_{R/2}(y_0)\setminus B_{R/4}(y_0)\big).
$$
Thus from the fact that
$$
\int_{B_{R/2}(y_0)} \operatorname{div}(\eta \phi)\,dx=0,
$$
we get
\begin{align*}
\int_\Omega \pi\operatorname{div}(\eta\phi)\,dx&=\int_{B_{R/2}(y_0)} \pi I_{\Omega}\operatorname{div}(\eta \phi)\,dx\\
&=\int_{B_{R/2}(y_0)} \tilde{\pi} \operatorname{div}(\eta \phi)\,dx=\int_{B_{R/2}(y_0)\setminus B_{R/4}(y_0)}\tilde{\pi}\operatorname{div}(\eta\phi)\,dx.
\end{align*}
Combining \eqref{180627@eq1} and \eqref{180627@eq2}, and using H\"older's inequality and \eqref{170711_eq1}, we see that \eqref{180202@eq5a} holds.
The lemma is proved.
\end{proof}

We now prove Theorem \ref{T2}.

\begin{proof}[Proof of Theorem \ref{T2}]
We first note that by utilizing Assumptions \ref{A0}, \ref{A1}, and \ref{A2}, and following the steps used in deriving  \eqref{170607@eq8}, we have
\begin{equation}		\label{170705@eq4}
\|G_\varepsilon(\cdot,y) \|_{L_1(\Omega_R(x))}\le CR^2
\end{equation}
for all $x,y\in \Omega$ and $0<2\varepsilon<R<R_0$, where $C=C(d,\lambda,K_0, A_0,A_1)$.

To prove \eqref{180115@A1} and the boundedness of $G(x,y)$  for $|x-y| \geq R_0$, we set
$$
R = \left\{
\begin{aligned}
|x-y|/2 &\quad \text{if} \quad |x-y| < R_0,
\\
R_0/2 &\quad \text{otherwise}
\end{aligned}
\right.
$$
for $x,y\in \Omega$ with $x\neq y$.
Let $0 < \sigma < R/2$.
Since $(G_\sigma^*(\cdot,x), \Pi^*_\sigma(\cdot,x))$ satisfies (see \eqref{171130@eq2})
$$
\left\{
\begin{aligned}
\operatorname{div}G^*_\sigma(\cdot,x)=0 &\quad \text{in }\, \Omega_R(y),\\
\cL^* G^*_\sigma(\cdot,x)+\nabla \Pi^*_\sigma(\cdot,x)=\frac{1}{|\Omega|}I &\quad \text{in }\, \Omega_R(y),\\
\cB^* G^*_\sigma(\cdot,x)+\nu\Pi^*_\sigma(\cdot,x)=0 &\quad \text{on }\, \partial \Omega\cap B_R(y),
\end{aligned}
\right.
$$
by Assumptions \ref{A1} and \ref{A2}, we have
$$
\|G_\sigma^*(\cdot,x)\|_{L_\infty(\Omega_{r/2}(z))}\le C\big(r^{-d/2}\|G^*_\sigma(\cdot,x)\|_{L_2(\Omega_r(z))}+R^2\big)
$$
for any $z\in \Omega_R(y)$ and $0<r<\operatorname{dist}(z, \partial B_R(y))$, where $C=C(d,m_0,A_0,A_1)$.
Then by the same reasoning as in \eqref{180621@eq1}, we obtain
\begin{align}		
\nonumber
\|G_\sigma^*(\cdot,x)\|_{L_\infty(\Omega_{R/2}(y))}
&\le C\big(R^{-d}\|G_\sigma^*(\cdot,x)\|_{L_1(\Omega_R(y))}+R^2\big)\\
\label{180202@eq1a}
&\le CR^{2-d},
\end{align}
where we used the counterpart of \eqref{170705@eq4} and $R\le  1/2$ in the second inequality.
Therefore, we get from \eqref{180208@A1} that
$$
|G(x,y)| \leq C R^{2-d}
$$
for $x, y \in \Omega$, $x \neq y$.
That is,
\begin{equation}
							\label{180202@eq2}
|G(x,y)| \leq
\left\{
\begin{aligned}
 C |x-y|^{2-d} &\quad \text{if} \quad |x-y|< R_0,
\\
C R_0^{2-d} &\quad \text{otherwise},
\end{aligned}
\right.
\end{equation}
where $C = C(d,\lambda,m_0,K_0,A_0,A_1)$.
In particular, the above inequality for $|x-y| < R_0$ proves \eqref{180115@A1}.

In the rest of the proof, we only prove the estimates $i)$ and $ii)$ in the theorem because the others are their easy consequences; see the proofs of Lemmas \ref{171130@lem1} and \ref{170710@lem2}.
Let $y\in \Omega$ and $0<16\varepsilon<R<R_0$.
If $\operatorname{dist}(y,\partial \Omega)> R/16$,  then by Lemma \ref{170710@lem1}, we have
\begin{equation}		\label{180202@eq3}
\begin{aligned}
&\|G_\varepsilon(\cdot,y)\|_{L_{2d/(d-2)}(\Omega\setminus B_{R}(y))}+\|DG_\varepsilon(\cdot,y)\|_{L_2(\Omega\setminus B_{R}(y))}\\
&\quad +\|\Pi_\varepsilon(\cdot,y)\|_{L_2(\Omega\setminus B_{R}(y))}\le CR^{(2-d)/2},
\end{aligned}
\end{equation}
where $C=C(d,\lambda,K_0,A_0)$.
In the case when $\operatorname{dist}(y,\partial \Omega)\le  R/16$, we take $y_0\in \partial \Omega$ such that $|y-y_0|=\operatorname{dist}(y, \partial \Omega)$ and
\begin{equation}		\label{170712@eq2}
B_{R/16}(y)\subset B_{R/8}(y_0)\subset B_{5R/8}(y_0)\subset B_R(y).
\end{equation}
Then by utilizing Lemma \ref{170708@lem1} and \eqref{170705@eq4}, and following the same argument used in deriving \eqref{171129@eq4}, we have
$$
\begin{aligned}
&\|(1-\eta)G_\varepsilon(\cdot,y)\|_{L_{2d/(d-2)}(\Omega)}+
\|(1-\eta)DG_\varepsilon(\cdot,y)\|_{L_2(\Omega)}\\
&\le \frac{C}{R}\|G_\varepsilon(\cdot,y)\|_{L_2(\Omega_{5R/8}(y_0)\setminus \Omega_{R/8}(y_0))}+C\\
&\le \frac{C}{R}\|G_\varepsilon(\cdot,y)\|_{L_2(\Omega_{R}(y)\setminus \Omega_{R/16}(y))}+C,
\end{aligned}
$$
where we used \eqref{170712@eq2} in the second inequality and $C=C(d,\lambda, m_0, K_0, A_0,A_1,\theta)$.
Hence, by applying the counterpart of \eqref{180202@eq1a}, we obtain
\begin{equation}		\label{180625@eq4}
\|(1-\eta)G_\varepsilon(\cdot,y)\|_{L_{2d/(d-2)}(\Omega)}+
\|(1-\eta)DG_\varepsilon(\cdot,y)\|_{L_2(\Omega)}\le CR^{(2-d)/2}.
\end{equation}
On the other hand, to estimate $\Pi_{\varepsilon}(\cdot,y)$, we follow the argument as in \eqref{180625@eq4a} to get
$$
\begin{aligned}
\|\Pi_\varepsilon(\cdot,y)\|_{L_2(\Omega\setminus B_{R/2}(y_0))}&\le C\|DG_\varepsilon(\cdot,y)\|_{L_2(\Omega\setminus B_{R/4}(y_0))}\\
&\quad +C\|{\Pi}_\varepsilon(\cdot,y)\|_{L_{2}(\Omega_{R/2}(y_0)\setminus B_{R/4}(y_0)) }+C.
\end{aligned}
$$
Then by Lemma \ref{170708@lem1}, \eqref{170712@eq2},  and \eqref{180625@eq4}, we have
\begin{align}
\nonumber
&\|\Pi_\varepsilon(\cdot,y)\|_{L_2(\Omega\setminus B_{R}(y))}\\
\nonumber
&\le C\|DG_\varepsilon(\cdot,y)\|_{L_2(\Omega\setminus B_{R/16}(y))}+\|G_\varepsilon(\cdot,y)\|_{L_{2d/(d-2)}(\Omega\setminus B_{R/16}(y))}+C\\
\label{180625@eq5}
&\le CR^{(2-d)/2}.
\end{align}
Therefore, by combining \eqref{180202@eq3}, \eqref{180625@eq4}, and \eqref{180625@eq5},  and applying the weak lower semi-continuity, we conclude that the estimates $i)$ and $ii)$ in the theorem hold.
\end{proof}

\section{Appendix}		\label{S6}

In this appendix, we give examples when Assumptions \ref{A1} and \ref{A2} hold.
For $x=(x_1,x')\in \bR^d$ and $r>0$, we denote
$$
B'_r(x')=\{y'\in \bR^{d-1}:|x'-y'|<r\}.
$$

\begin{assumption}[$\gamma$]		\label{170219@ass2}
There exists $R_0\in (0,1]$ such that the following hold.
\begin{enumerate}[$(i)$]
\item
For $z\in \Omega$ and $0<R\le \min\{R_0,\operatorname{dist}(z,\partial \Omega)\}$, there exists a coordinate system depending on $z$ and $R$ such that in this new coordinate system, we have that
\begin{equation}		\label{170219@eq1}
\dashint_{B_R(z)}\bigg|A^{\alpha\beta}(x_1,x')-\dashint_{B'_R(z')}A^{\alpha\beta}(x_1,y')\,dy'\bigg|\,dx\le \gamma.
\end{equation}
\item
For any $z\in \partial \Omega$ and $0<R\le R_0$, there is a coordinate system depending on $z$ and $R$ such that in the new coordinate system we have that \eqref{170219@eq1} holds, and
$$
\{y:z_1+\gamma R<y_1\}\cap B_R(z)\subset \Omega_R(z)\subset \{y:z_1-\gamma R<y_1\}\cap B_R(z),
$$
where $z_1$ is the first coordinate of $z$ in the new coordinate system.
\end{enumerate}
\end{assumption}

In the theorem below, we prove that Assumption \ref{A1} holds when the coefficients $A^{\alpha\beta}$ of $\cL$ are variably partially BMO satisfying  Assumption \ref{170219@ass2} $(\gamma)$ $(i)$.

\begin{theorem}		\label{171206@thm1}
Let $d \ge 3$ and $\Omega$ be a bounded domain in $\bR^d$.
There exists a constant $\gamma=\gamma(d,\lambda)\in (0,1]$ such that, under Assumption \ref{170219@ass2} $(\gamma)$ $(i)$, the following holds:
Let $x_0\in \Omega$, $0<R<\min\{R_0, \operatorname{dist}(x_0, \partial \Omega)\}$, and $(u, p)\in W^1_2(B_R(x_0))^d\times L_2(B_R(x_0))$ satisfy \eqref{171129@eq1}.
Then we have $u\in C(\overline{B_{R/2}(x_0)})^d$ with the estimate
$$
\|u\|_{L_\infty(B_{R/2}(x_0))}\le A_0 \big(R^{-d/2}\|u\|_{L_2(B_R(x_0))}+R^2\|f\|_{L_\infty(B_R(x_0))}\big),
$$
where $A_0=A_0(d,\lambda)$.
\end{theorem}

\begin{proof}
The proof is based on a bootstrap argument and the $L_q$-estimate of the conormal derivative problem for the Stokes system.
Without loss of generality, we assume that $x_0 = 0$ and use the abbreviation $B_r=B_r(0)$.
Assume that
$(u, p)\in W^1_2(B_r)^d\times L_2(B_r)$ satisfies
$$
\left\{
\begin{aligned}
\operatorname{div} u=0 \quad \text{in }\, B_r,\\
\cL u+\nabla p=f \quad \text{in }\, B_r,
\end{aligned}
\right.
$$
where $0<r<\min\{R_0, \operatorname{dist}(0, \partial \Omega)\}$, and $f\in L_\infty(B_r)^d$.
For $\rho\in (0,r)$, let $\eta$ and $\zeta$ be infinitely differentiable functions on $\bR^d$ satisfying
$$
0\le \eta\le 1, \quad \eta\equiv 1 \, \text{ on }\, B_{\rho}, \quad \operatorname{supp}\eta\subset B_{\rho_1}, \quad |\nabla \eta|\le 4\sqrt{d}(r-\rho)^{-1},
$$
$$
0\le \zeta\le 1, \quad \zeta\equiv 1 \, \text{ on }\, B_{\rho_1}, \quad \operatorname{supp}\zeta\subset B_{\rho_2}, \quad |\nabla \zeta|\le 8\sqrt{d}(r-\rho)^{-1},
$$
where $\rho_1=\frac{r+\rho}{2}$ and $\rho_2=\frac{\rho_1+r}{2}$.
Set $\tilde{\cL}$ to be an operator of the form
$$
\tilde{\cL} u=D_\alpha(\tilde{A}^{\alpha\beta}D_\beta u),
$$
where $\tilde{A}^{\alpha\beta}_{ij}=\zeta A^{\alpha\beta}_{ij}+(1-\zeta)\delta_{\alpha\beta}\delta_{ij}$.
Here $\delta_{ij}$ is the Kronecker delta symbol.
Then $\tilde{A}^{\alpha\beta}$ satisfy the strong ellipticity condition \eqref{171209@eq1}.
Note that
$$
\tilde{A}^{\alpha\beta}_{ij}(x_1,x')-\tilde{A}^{\alpha\beta}_{ij}(x_1,y')
$$
$$
= \zeta(x_1,x') \left(A_{ij}^{\alpha\beta}(x_1,x') - A_{ij}^{\alpha\beta}(x_1,y')\right) + A_{ij}^{\alpha\beta}(x_1,y') \left(\zeta(x_1,x') - \zeta(x_1,y') \right)
$$
$$
+ \left(\zeta(x_1,y') - \zeta(x_1,x')\right) \delta_{\alpha\beta} \delta_{ij}.
$$
Under Assumption \ref{170219@ass2} ($\gamma$) (i), this implies that
\begin{equation}
							\label{eq0328_04}
\dashint_{B_R(z)}\bigg|\tilde{A}^{\alpha\beta}(x_1,x')-\dashint_{B'_R(z')}\tilde{A}^{\alpha\beta}(x_1,y')\,dy'\bigg|\,dx\le \gamma + \frac{C_0R}{r-\rho}
\end{equation}
for all $z \in \Omega$ and $0 < R \leq \min\{R_0, \operatorname{dist}(z,\partial\Omega)\}$, where $C_0=C_0(d,\lambda)\ge 1$.

Set $(v,\pi):=\big(\eta (u-(u)_{B_r}),\eta p\big)$, which satisfies
$$
\left\{
\begin{aligned}
\operatorname{div} v= G &\quad \text{in }\, B_r,\\
\tilde{\cL} v+\nabla \pi=F+D_\alpha F_\alpha &\quad \text{in }\, B_r,\\
\tilde{\cB} v+\pi \nu=F_\alpha \nu_\alpha &\quad \text{on }\, \partial B_r,
\end{aligned}
\right.
$$
where
$$
G = \nabla \eta\cdot (u-(u)_{B_r}),
$$
$$
F=\eta f+A^{\alpha\beta}D_\beta u D_\alpha \eta+p \nabla \eta, \quad F_\alpha=A^{\alpha\beta}D_\beta \eta (u-(u)_{B_r}).
$$
We now obtain an $L_q$-estimate for this system.
To see clearly the parameters on which the constant in the estimate depends, let us rescale as follows.
Set
$$
\bar{v}(x) = r^{-2}v(rx), \quad \bar{\pi}(x) = r^{-1} \pi(rx),
$$
$$
\bar{G}(x) = r^{-1}G(rx), \quad \bar{F}(x) = F(rx), \quad \bar{F}_\alpha(x) = r^{-1}F_\alpha(rx),
$$
$$
\bar{A}^{\alpha\beta}(x) = \tilde{A}^{\alpha\beta}(rx).
$$
Then $(\bar{v}, \bar{\pi})$ satisfies
\begin{equation}		\label{eq0328_02}
\left\{
\begin{aligned}
\operatorname{div} \bar{v}= \bar{G} &\quad \text{in }\, B_1,\\
\bar{\cL} \bar{v}+\nabla \bar{\pi}=\bar{F}+D_\alpha \bar{F}_\alpha &\quad \text{in }\, B_1,\\
\bar{\cB} \bar{v}+\bar{\pi} \nu=\bar{F}_\alpha \nu_\alpha &\quad \text{on }\, \partial B_1,
\end{aligned}
\right.
\end{equation}
where
$$
\bar{\cL} u = D_\alpha (\bar{A}^{\alpha\beta}D_\beta u).
$$

For given $q>1$, let $\gamma_q=\gamma_q(d,\lambda,q)\in (0,1/48]$ be the constant from \cite[Theorem 2.2]{MR3809039}.
We take $\gamma \in (0,\gamma_q/2]$ and set
$$
R_1=\frac{\gamma_q}{2C_0}\frac{r-\rho}{r},
$$
where $C_0=C_0(d,\lambda)\ge 1$ is the constant in \eqref{eq0328_04}.
Then, under Assumption \ref{170219@ass2} $(\gamma)$ $(i)$, we have the following.
\begin{enumerate}[$(a)$]
\item
For any $z\in B_1$ and $0<\tau \le \min\{R_1, \operatorname{dist}(z,\partial B_1)\}$,
there exists a coordinate system depending on $z$ and $\tau$ such that in this new coordinate system, we have that
\begin{equation}		\label{eq0328_03}
\dashint_{B_\tau (z)}\bigg|\bar{A}^{\alpha\beta}(x_1,x')-\dashint_{B'_\tau(z')} \bar{A}^{\alpha\beta}(x_1,y')\,dy'\bigg|\,dx\le \gamma_q.
\end{equation}
\item
For any $z\in \partial B_1$ and $0< \tau \le R_1$, there is a coordinate system depending on $z$ and $\tau$ such that in the new coordinate system we have that \eqref{eq0328_03} holds, and
$$
\{y:z_1+\gamma_q \tau<y_1\}\cap B_\tau (z) \subset B_1 \cap B_\tau (z)\subset \{y:z_1-\gamma_q\tau <y_1\}\cap B_\tau (z).
$$
\end{enumerate}
Indeed, the statement (a) follows from \eqref{eq0328_04} with scaling, and the statement (b) is due to the facts that $\bar{A}^{\alpha\beta}$ is the Laplace operator near the boundary of $B_1$ and the flatness of $\partial B_1$ can be controlled by the radius of the ball $B_\tau(z)$.

Since $(\bar{v}-(\bar{v})_{B_1}, \bar{\pi})$ satisfies the same system \eqref{eq0328_02}, by \cite[Theorem 2.2]{MR3809039} applied to \eqref{eq0328_02}, we have
\begin{equation}
							\label{eq0328_05}
\|D\bar{v}\|_{L_q(B_1)} + \| \bar{\pi}\|_{L_q(B_1)} \leq C \big(\|\bar{F} \|_{L_{qd/(q+d)}(B_1)} + \|\bar{F}_\alpha\|_{L_q(B_1)}+\|\bar{G}\|_{L_q(B_1)}\big),
\end{equation}
where
$$
C = C(d,\lambda,q,R_1,\operatorname{diam}B_1) = C(d,\lambda,q,(r-\rho)/r).
$$
Thus, from \eqref{eq0328_05} with scaling and the following Poincar\'e inequality
$$
\|u-(u)_{B_r}\|_{L_q(B_r)}\le Cr\|Du\|_{L_{qd/(q+d)}(B_r)},
$$
we get (using $r\le 1$)
\begin{align}
\nonumber
&\|Du\|_{L_{q}(B_{\rho})}+\|p\|_{L_{q}(B_{\rho})}\\
\nonumber
&\le \|Dv\|_{L_q(B_r)}+\|\pi\|_{L_q(B_r)}\\
\label{171213@eq2}
&\le \frac{C}{r-\rho}\big(\|Du\|_{L_{qd/(q+d)}(B_r)}+\|p\|_{L_{qd/(q+d)}(B_r)}\big)+C\|f\|_{L_{qd/(q+d)}(B_r)},
\end{align}
where $C=C(d,\lambda,q, (r-\rho)/r)$.

We are ready to complete the proof of the theorem.
Let $x_0\in \Omega$, $0<R<\min\{R_0, \operatorname{dist}(x_0, \partial \Omega)\}$, and $(u,p)\in W^1_2(B_R(x_0))^d\times L_2(B_R(x_0))$ satisfy \eqref{171129@eq1}.
We fix $q>d$ and set
$$
q_i=\frac{dq}{d+qi}, \quad r_i=\frac{R}{4}\left(1+\frac{i}{N}\right), \quad i\in \{0, \ldots, N\},
$$
where $N$ is the smallest integer such that $q_N\le 2$.
We take $\gamma=\gamma(d,\lambda,q)\in (0, \gamma_q/2]$ satisfying the above properties $(a)$ and $(b)$, where $\gamma_q=\min\{\gamma_{q_0}, \ldots, \gamma_{q_N}\}$ and $\gamma_{q_i}=\gamma_{q_i}(d,\lambda,q_i)$ are the constants from \cite[Theorem 2.2]{MR3809039}.
Note that $(u, p-(p)_{B_{R/2}(x_0)})$ satisfies the same system \eqref{171129@eq1}.
Under Assumption \ref{170219@ass2} $(\gamma)$ $(i)$,
by applying \eqref{171213@eq2} iteratively, we get
$$
\begin{aligned}
&\|Du\|_{L_{q}(B_{R/4}(x_0))}+\|p-(p)_{B_{R/2}(x_0)}\|_{L_{q}(B_{R/4}(x_0))}\\
&\le \left(\frac{C}{R}\right)^N \big(\|Du\|_{L_{q_{N}}(B_{R/2}(x_0))}+\|p-(p)_{B_{R/2}(x_0)}\|_{L_{q_{N}}(B_{R/2}(x_0))}\big)\\
&\quad +C\sum_{i=0}^{N-1}\left(\frac{C}{R}\right)^{i}\|f\|_{L_{q_{i+1}}(B_{r_{i+1}}(x_0))},
\end{aligned}
$$
where $C=C(d,\lambda, q)$.
Therefore, by H\"older's inequality, we have
$$
\begin{aligned}
\|Du\|_{L_{q}(B_{R/4}(x_0))}&\le CR^{d/q-d/2} \big(\|Du\|_{L_{2}(B_{R/2}(x_0))}+\|p-(p)_{B_{R/2}(x_0)}\|_{L_{2}(B_{R/2}(x_0))}\big)\\
&\quad +CR^{1+d/q}\|f\|_{L_{\infty}(B_R(x_0))}\\
&\le CR^{d/q-d/2-1}\|u\|_{L_{2}(B_R(x_0))}+CR^{1+d/q}\|f\|_{L_{\infty}(B_R(x_0))},
\end{aligned}
$$
where we used Caccioppoli's inequality (see Lemma \ref{171205@lem5}) in the second inequality.
Since $q>d$, by using Morrey's inequality, we have that $u\in C^{1-d/q}(\overline{B_{R/4}(x_0)})^d$ and
$$
\|u\|_{L_\infty(B_{R/4}(x_0))}\le CR^{-d/2}\|u\|_{L_2(B_R(x_0))}+R^2\|f\|_{L_\infty(B_R(x_0))},
$$
where $C=C(d,\lambda,q)$.
Finally, using a covering argument, we easily see that the desired estimate holds.
The theorem is proved.
\end{proof}

In the theorem below, we prove that Assumption \ref{A2} holds when the coefficients of $\cL$ are variably partially BMO and the domain is Reifenberg flat.

\begin{theorem}
Let $d \ge 3$ and $\Omega$ be a bounded domain in $\bR^d$ satisfying $\operatorname{diam}\Omega\le M_0$.
There exists a constant $\gamma=\gamma(d,\lambda)\in (0,1]$ such that, under Assumption \ref{170219@ass2}, the following holds:
Let $x_0\in \partial\Omega$, $0<R<R_0$, and $(u, p)\in W^1_2(\Omega_R(x_0))^d\times L_2(\Omega_R(x_0))$ satisfy \eqref{171129@eq1a}.
Then we have $u\in C(\overline{\Omega_{R/2}(x_0)})^d$ with the estimate
\begin{equation}		\label{171208@A2}
\|u\|_{L_\infty(\Omega_{R/2}(x_0))}\le A_1 \big(R^{-d/2}\|u\|_{L_2(\Omega_R(x_0))}+R^2\|f\|_{L_\infty(\Omega_R(x_0))}\big),
\end{equation}
where $A_1=A_1(d,\lambda, M_0,R_0)$.
\end{theorem}

\begin{proof}
The proof of the theorem is similar to that of Theorem \ref{171206@thm1}.
Let $(u, p)\in W^1_2(\Omega_R(x_0))^d\times L_2(\Omega_R(x_0))$ satisfy \eqref{171129@eq1a}, where  $x_0\in \partial\Omega$ and $0<R<R_0$.
Without loss of generality, we assume that $x_0=0$.
We denote $B_\tau=B_\tau(x_0)$ and $\Omega_\tau=\Omega_\tau(x_0)$.
For $0<\rho<r\le R$, let $\eta$ be an infinitely differentiable function on $\bR^d$ satisfying
$$
0\le \eta\le 1, \quad \eta\equiv 1 \, \text{ on }\, B_{\rho}, \quad \operatorname{supp}\eta\subset B_{r}, \quad |\nabla \eta|\le C(d)(r-\rho)^{-1}.
$$
Then  $(v,\pi):=\big(\eta (u-(u)_{\Omega_{R/4})},\eta p\big)$ satisfies
\begin{equation}		\label{171206@EQ1}
\left\{
\begin{aligned}
\operatorname{div}v=\nabla \eta \cdot (u-(u)_{\Omega_{R/4}})\quad &\text{in }\ \Omega,\\
\cL v+\nabla \pi= F +D_\alpha F_\alpha \quad &\text{in }\ \Omega,\\
\cB v+\pi\nu =F_\alpha \nu_\alpha \quad &\text{on }\ \partial \Omega,
\end{aligned}
\right.
\end{equation}
where
$$
F=\eta f+A^{\alpha\beta}D_\beta u  D_\alpha \eta + p \nabla \eta, \quad F_\alpha=A^{\alpha\beta}D_\beta \eta (u-(u)_{\Omega_{R/4}}).
$$
Note that $(v-(v)_\Omega, \pi)$ satisfies the same problem \eqref{171206@EQ1}.
Therefore, for given $q>1$, under Assumption \ref{170219@ass2} $(\gamma)$, where $\gamma=\gamma(d,\lambda,q)$ is the constant from \cite[Theorem 2.2]{MR3809039}, we have that
\begin{equation}		\label{171213@eq3}
\begin{aligned}
&\|Du\|_{L_{q}(\Omega_\rho)}+\|p\|_{L_{q}(\Omega_\rho)}\\
&\le \|Dv\|_{L_q(\Omega)}+\|\pi\|_{L_q(\Omega)}\\
&\le \frac{C}{r-\rho}\big(\|Du\|_{L_{qd/(q+d)}(\Omega_{r})}+\|p\|_{L_{qd/(q+d)}(\Omega_{r})}\big)\\
&\quad +\frac{C}{r-\rho}\| (u-(u)_{\Omega_{R/4}})\|_{L_{q}(\Omega_{r})}+C\|f\|_{L_{qd/(q+d)}(\Omega_{r_{i+1}})},
\end{aligned}
\end{equation}
where $C=C(d,\lambda,q,R_0,M_0)$.

Now we fix $q>d$, and set
$$
q_i=\frac{dq}{d+qi}, \quad r_i=\frac{R}{8}\left(1+\frac{i}{N}\right), \quad i\in \{0,1,\ldots,N\},
$$
where $N$ is the smallest integer such that $q_{N}\le 2$.
Set
$$
\gamma=\min\{\gamma_0,\ldots,\gamma_{N}\}\in (0,1/48],
$$
where $\gamma_i=\gamma_i(d,\lambda,q_i)$ are constants from
\cite[Theorem 2.2]{MR3809039}.
Then under Assumption \ref{170219@ass2} $(\gamma)$, we get from \eqref{171213@eq3} that
$$
\begin{aligned}
&\|Du\|_{L_{q}(\Omega_{R/8})}+\|p\|_{L_{q}(\Omega_{R/8})}\\
&\le \left(\frac{C}{R}\right)^{N}\big(\|Du\|_{L_{q_{N}}(\Omega_{R/4})}+\|p\|_{L_{q_{N}}(\Omega_{R/4})}\big)\\
&\quad +\sum_{i=0}^{N-1}\left(\frac{C}{R}\right)^{i+1}\|u-(u)_{\Omega_{R/4}}\|_{L_{q_i}(\Omega_{r_{i+1}})}+C\sum_{i=0}^{N-1} \left(\frac{C }{R}\right)^i\|f\|_{L_{q_{i+1}}(\Omega_{r_{i+1}})}.
\end{aligned}
$$
Therefore, by H\"older's inequality, we obtain that
\begin{align}
\nonumber
\|Du\|_{L_{q}(\Omega_{R/8})}
&\le CR^{d/q-d/2}\big(\|D u\|_{L_{2}(\Omega_{R/4})}+\|p\|_{L_{2}(\Omega_{R/4})}\big)\\
\nonumber
&\quad +\frac{C}{R}\|u-(u)_{\Omega_{R/4}}\|_{L_{q}(\Omega_{R/4})}+CR^{1+d/q}\|f\|_{L_{\infty}(\Omega_{R})}\\
\label{171208@A1}
&\le CR^{d/q-d/2-1}\|u\|_{L_{2}(\Omega_{R})}+CR^{1+d/q}\|f\|_{L_{\infty}(\Omega_{R})},
\end{align}
where we used the Poincar\'e inequality (see \cite[Theorem 3.5]{MR3809039}) and Caccioppoli's inequality (see Lemma \ref{171205@lem5} $(b)$) in the second inequality.
Since $q>d$, the inequality \eqref{171208@A1} implies that $u$ is H\"older continuous, and thus it is bounded in $\Omega_r$ for any $r<R/8$.

We finally show \eqref{171208@A2}, the proof of which is more involved than the corresponding estimate in Theorem \ref{171206@thm1}.
For a proof, one may consider an extension of $u$ using the fact that Reifenberg domains are extension domains.
However, we have a ball intersected with a Reifenberg flat domain, which may not share the same nice properties as the domain $\Omega$ (as an extension domain).
To deal with such an intersection  and obtain the precise information about the parameters on which the constant $A_1$ depends, we proceed as follows. Let
$$
z_0=(R/32,0,\ldots,0)
$$
in the coordinate system associated with the point $x_0=0$ and $R/4$ satisfying
$$
\{y:\gamma R/4<y_1\} \cap B_{R/4} \subset \Omega_{R/4} \subset \{y:-\gamma R/4<y_1\}\cap B_{R/4}.
$$
Then by the proof of \cite[Theorem 3.5]{MR3809039} (see the inequality below \cite[Eq. (57)]{MR3809039}), we  see that
$$
|u(x)-\bar{u}|\le C(d)\int_{\Omega_{R/8}} \frac{|D u(y)|}{|x-y|^{d-1}}\,dy
$$
for any $x\in \Omega_{R/16}$.
Here, we set
$$
\bar{u}=\frac{1}{\|\phi\|_{L_1(\bR^d)}}\int_{\bR^d} u(z)\phi(z-z_0)\,dz,
$$
where $\phi$ is a smooth function in $\bR^d$ satisfying
$$
0\le \phi\le 1, \quad \phi\equiv 1 \, \text{ on }\, B_{R/32h}, \quad \operatorname{supp} \phi \subset B_{R/16h}, \quad h=4\cdot 24\cdot 48.
$$
By H\"older's inequality, we have
$$
\begin{aligned}
\|u\|_{L_\infty(\Omega_{R/16})}&\le |\bar{u}|+\|u-\bar{u}\|_{L_\infty(\Omega_{R/16})}\\
&\le C R^{-d/2}\|u\|_{L_2(\Omega_R)}+ CR^{1-d/q}\|Du\|_{L_q(\Omega_{R/8})},
\end{aligned}
$$
and thus, we get from \eqref{171208@A1} that
$$
\|u\|_{L_\infty(\Omega_{R/16})}\le CR^{-d/2}\|u\|_{L_2(\Omega_R)}+CR^2 \|f\|_{L_\infty(\Omega_R)},
$$
where $C=C(d,\lambda,q,R_0,M_0)$.
Finally, based on a covering argument, using  the above inequality and Theorem \ref{171206@thm1}, we get the desired estimate.
\end{proof}

\section*{Acknowledgments} The authors would like to thank the referees for careful reading and helpful comments.

\bibliographystyle{plain}
\def\cprime{$'$}

\end{document}